\documentclass[12pt,reqno]{amsart}   
\usepackage{lipsum}
\usepackage{graphicx}
\usepackage{array}
\usepackage{amssymb}
\usepackage{hyperref}
\usepackage{amsthm}
\usepackage{amsfonts}
\usepackage{amsmath}
\usepackage{bm}
\usepackage{mathrsfs}
\usepackage[all]{xy}
\usepackage{color}
\usepackage{subfigure}
\usepackage{tikz, tikz-cd}
\usepackage{enumerate}
\usepackage{anysize}
\usepackage{amscd}
\usepackage[letterpaper]{geometry}
\usepackage{geometry}
\usepackage{multirow}
\usepackage{array}
\usepackage{booktabs}
\usepackage{musicography}
\usepackage[utf8]{inputenc}
\usepackage{slashed}
\usepackage{comment}

\newtheorem{theorem}{Theorem}[section]
\newtheorem{proposition}[theorem]{Proposition}

\newtheorem{lemma}[theorem]{Lemma}

\theoremstyle{remark}

\theoremstyle{definition}
\newtheorem{definition}[theorem]{Definition}
\newtheorem{remark}[theorem]{Remark}

\numberwithin{equation}{section}
\numberwithin{figure}{section}
\numberwithin{table}{section}
\numberwithin{equation}{section}
\numberwithin{figure}{section}
\numberwithin{table}{section}

\def\GG{\mathrm{G}}
\def\d{\mathrm{d}}
\DeclareMathOperator\Spin{Spin}

\allowdisplaybreaks[3]

\begin{document}

\title{Geometric flows of  G\textsubscript{2}-structures on
3-Sasakian 7-manifolds}

\author{Aaron Kennon}

\address{Department of Physics UCSB}
\email{akennon@physics.ucsb.edu}

\author{Jason D. Lotay}
\address{Mathematical Institute, University of Oxford}
\email{lotay@maths.ox.ac.uk}

\date{\today}

\thanks{The research of the first 
author is partially supported by Simons Foundation Award \#488629 (Morrison).  The research of the second author is partially supported by the Simons Collaboration on Special Holonomy in Geometry, Analysis, and Physics (\#724071 Jason Lotay).
}

%On a 3-Sasakian 7-manifold there are two distinct families of $\GG_2$-structures containing nearly parallel $\GG_2$-structures that define the  natural Einstein metrics on the 3-Sasakian 7-manifold: the 3-Sasakian metric and the so-called  ``squashed'' Einstein metric. 

\begin{abstract}
A 3-Sasakian structure on a 7-manifold may be used to define two distinct Einstein metrics: the  3-Sasakian metric and the squashed Einstein metric. Both  metrics are induced by nearly parallel $\GG_2$-structures which may also be expressed in terms of the 3-Sasakian structure. Just as Einstein metrics are  critical points for the Ricci flow up to rescaling, nearly parallel $\GG_2$-structures provide natural critical points of the (rescaled) geometric flows of $\GG_2$-structures known as the Laplacian flow and Laplacian coflow.  We study each of these flows in the 3-Sasakian setting and see that their behaviour is markedly different, particularly regarding the stability of the nearly parallel $\GG_2$-structures.  We also compare the behaviour of the flows of $\GG_2$-structures with the (rescaled) Ricci flow.
\end{abstract}

\maketitle
\setcounter{tocdepth}{1}
\tableofcontents

\section{Introduction}

\subsection{Nearly parallel \texorpdfstring{G\textsubscript{2}}{G2}-structures} A $\GG_2$-structure on a 7-manifold is encoded by a 3-form $\varphi$ satisfying a certain nondegeneracy condition, and such a 3-form determines a Riemannian metric and orientation.  One of the most important types of  $\GG_2$-structure is a \emph{nearly parallel} $\GG_2$-structure
%, one for which
%\begin{equation}
%    \d\varphi=\lambda *_{\varphi}\varphi
%\end{equation}
%
since it defines an Einstein metric with positive scalar curvature, as well as a real Killing spinor \cite{Bar,FKMS}.  Moreover, the cone over a 7-manifold with a nearly parallel $\GG_2$-structure admits a conical metric with exceptional holonomy $\Spin(7)$ (and so is Ricci-flat), and thus nearly parallel $\GG_2$-structures are also important in the study of asymptotically conical and conically-singular $\Spin(7)$ manifolds (cf.~\cite{Lehmann}).

%One of the most interesting classes of G$_{2}$-Structures is the Nearly Parallel (NP) class. These structures define Einstein metrics with positive scalar curvature as well as real Killing spinors \cite{FKMS,Bar}. As the cone over a seven-manifold with a NP G$_{2}$-Structure has holonomy contained in Spin(7), these manifolds are also important for understanding asymptotically conical and conically singular Spin(7)-Manifolds \cite{Lehmann}. \\

Since the existence of a complete positive Einstein metric will lead to compactness of the underlying manifold by Myers theorem, it is natural to ask which compact 7-manifolds admit nearly parallel $\GG_2$-structures.   Though this general question is currently open, an infinite number of examples of such compact 7-manifolds are known, including the 7-sphere, the Aloff--Wallach spaces $N(k,l)$, the Berger space SO(5)/SO(3), and the Stiefel manifold V$_{5,2}$ \cite{FKMS}.  The largest class of 7-manifolds that are known to admit nearly parallel $\GG_2$-structures are the \emph{$3$-Sasakian} 7-manifolds, which are the focus of this paper.  %As the metric corresponding to a nearly G$_{2}$-Structure is Einstein with positive scalar curvature, Myer's theorem implies that if the metric is also complete then the underlying manifold is compact. %Moreover, Bochner's theorem implies that the manifold admits no harmonic one-forms. \\

\subsection{Geometric flows} 
Nearly parallel $\GG_2$-structures are natural to study from the perspective of several geometric flows. Since a nearly parallel $\GG_2$-structure induces a positive Einstein metric, it is natural to evolve its induced metric $g$ by the \emph{Ricci flow}:
\begin{equation}\label{eq:Ric.flow}
\frac{\partial g}{\partial t}=-2\mathrm{Ric}(g).
\end{equation}
The induced metric will define a self-similarly \emph{shrinking} solution to the Ricci flow, and thus a critical point after rescaling.  %As the metric of a NP G$_{2}$-Structure is Einstein, it is a special type of soliton solution for the Ricci flow. Moreover, as the scalar curvature is positive it follows from standard theory that the Ricci flow will go extinct in finite time. \\
%
%For a G$_{2}$-Structure the metric may be defined algebraically in terms of a differential three-form, or alternatively in terms of a related four-form. For such structures then we can consider refined flows which modify the underlying differential forms and then determine flows on the metric. \\
%
However, a $\GG_2$-structure contains more information than the metric (since the same metric is induced by a whole family of $\GG_2$-structures), so it is worthwhile to examine flows of $\GG_2$-structures relevant to nearly parallel $\GG_2$-structures, and compare and contrast its behaviour to the Ricci flow.

Two such geometric flows of $\GG_2$-structures which have been the most studied, and we shall examine here, are the \emph{Laplacian flow} (introduced by Bryant \cite{Bryant}) and the \emph{Laplacian coflow} (first considered in \cite{KMT}\footnote{It should be noted that  in \cite{KMT} the opposite sign for the velocity of the Laplacian coflow is used.}).  

\subsubsection{Laplacian flow} The Laplacian flow evolves the 3-form $\varphi$ defining the $\GG_2$-structure by its Hodge Laplacian:
\begin{equation}\label{eq:Lap.flow}
\frac{\partial\varphi}{\partial t} = \Delta_\varphi \varphi=(\d\d^*_{\varphi}\varphi+\d^*_{\varphi}\d)\varphi.
\end{equation}
(Here, we emphasise the nonlinearity in the formal adjoint $\d^*_{\varphi}$ of the exterior derivative, since the metric and orientation depend on $\varphi$.)
The Laplacian flow has received particular attention in the context of \emph{closed} $\GG_2$-structures (when $\d\varphi=0$), where it has many attractive features, particularly with regards to torsion-free $\GG_2$-structures (when $\d\varphi=0$ and $\d^*_{\varphi}\varphi=0$), which define Ricci-flat metrics with holonomy contained in $\GG_2$.  For foundational results and a survey of recent developments in the Laplacian flow for closed $\GG_2$-structures see e.g.~\cite{KLL,LotayWei1,LotayWei2}.

A nearly parallel $\GG_2$-structure defines a self-similarly \emph{expanding} solution to the Laplacian flow \eqref{eq:Lap.flow}, so can be viewed as a critical point up to rescaling.  (We note the immediate difference with the Ricci flow where the induced metric was a shrinker.) 
A nearly parallel $\GG_2$-structure is, however, not closed but \emph{coclosed}: the defining 3-form $\varphi$ satisfies $\d^*_{\varphi}\varphi=0$. Whilst it may seem potentially plausible to study coclosed $\GG_2$-structures using the Laplacian flow \eqref{eq:Lap.flow}, in fact it is not yet known in general whether this flow even has short time existence starting at a coclosed $\GG_2$-structure.  An example situation where it has proved instructive to use the Laplacian flow to study coclosed $\GG_2$-structures can be found in \cite{LSS}.  

\subsubsection{Laplacian coflow} Currently the best candidate\footnote{The Laplacian coflow for coclosed $\GG_2$-structures has many attractive features analogous to the Laplacian flow for closed $\GG_2$-structures, but with the significant difference that the analytic foundations for the Laplacian coflow are currently lacking: see \cite{Grigorian,KLL} for a discussion of the  analytic issues.} for studying coclosed $\GG_2$-structures is the Laplacian coflow, which evolves the closed 4-form $\psi=*_{\varphi}\varphi$ dual to the 3-form $\varphi$ defining the $\GG_2$-structure by its Hodge Laplacian:
\begin{equation}\label{eq:Lap.coflow}
\frac{\partial \psi}{\partial t} = \Delta_\psi\psi=(\d^*_{\psi}\d+\d\d^*_{\psi})\psi=\d\d^*_{\psi}\psi,
\end{equation}
using the fact that $\psi$ is closed.  (The 4-form $\psi$ induces the metric just like $\varphi$, but not the orientation, though an orientation can be fixed by the initial choice of $\GG_2$-structure.)  The Laplacian coflow   preserves the cohomology class $[\psi]$ of $\psi$, where it may be viewed as the gradient flow of the Hitchin volume functional, and the induced flow of the metric $g$ defined by $\psi$ is
\begin{equation}\label{eq:Ric.flow.psi}
    \frac{\partial g}{\partial t}=-2\mathrm{Ric}(g)+Q(\d\varphi),
\end{equation}
where $Q$ is a quadratic expression in $\d\varphi$: see \cite{Grigorian,KLL} for details.  Since $Q$ only depends on first order information on $\psi$, whereas the Ricci tensor involves second order data, one may view \eqref{eq:Ric.flow.psi} as a lower order perturbation of the Ricci flow \eqref{eq:Ric.flow}.  

However, just as for the Laplacian flow, a nearly parallel $\GG_2$-structure defines a self-similarly \emph{expanding} solution to the Laplacian coflow \eqref{eq:Lap.coflow}, whereas its induced metric defines a shrinker for Ricci flow.  Hence the ``lower order terms'' in \eqref{eq:Ric.flow.psi} drastically alter the behaviour of the metric flow in this setting.
 
We should also note that coclosed $\GG_2$-structures satisfy a parametric h-principle (see \cite{CN}). Therefore, coclosed G$_{2}$-structures exist on any (compact or non-compact) 7-manifold admitting a $\GG_2$-structure, which just requires the 7-manifold to be oriented and spin, and so the Laplacian coflow can potentially be studied on any oriented spin 7-manifold.  By contrast, it is currently not clear how restrictive the closed condition is for a $\GG_2$-structure on a compact manifold. 

\subsection{3-Sasakian 7-manifolds} A 3-Sasakian 7-manifold is a Riemannian 7-manifold $M$ so that the metric cone over it is hyperk\"ahler. One can use the 3-Sasakian structure to define two\footnote{In fact, there are three natural nearly parallel $\GG_2$-structures inducing the 3-Sasakian metric, but these are permuted by the symmetries in the 3-Sasakian structure.  The same does \emph{not} occur for the squashed Einstein metric.} distinct nearly parallel $\GG_2$-structures (up to scale), one of which induces the original 3-Sasakian Einstein metric on $M$, and the other induces the so-called squashed Einstein metric on $M$.
This is most easily seen in the example of the 7-sphere, where the 3-Sasakian metric is the round metric, and the squashed Einstein metric is obtained by rescaling the 3-sphere fibres relative to the 4-sphere base in the Hopf fibration of the 7-sphere.    %(In fact, there are three natural nearly parallel $\GG_2$-structures inducing the 3-Sasakian metric, but these are permuted by the symmetries in the 3-Sasakian structure.)

\subsection{Stability} Our primary goal is to study the stability of   nearly parallel G$_{2}$-structures on 3-Sasakian 7-manifolds under the Laplacian flow and Laplacian coflow, and to compare the behavior of these flows to the Ricci flow near their induced Einstein metrics. 

For geometric flows, one is primarily interested in the question of dynamical stability of a critical point, i.e.~when the flow starting near a critical point will flow back to it.  An easier and weaker thing to check is linear stability:  whether the critical point is stable for the linearized flow at that point.   
In some situations, one can infer dynamical stability from linear stability: e.g.~for complete positive Einstein metrics in Ricci flow, linear stability plus an integrability assumption implies a weak form of dynamical stability \cite{Kroncke}. 

In the context of nearly parallel $\GG_2$-structures on 7-manifolds $M$, it was shown in \cite{SWW} that  if the third Betti number $b^3(M)\neq 0$, then under the Ricci flow any Einstein metric induced by a nearly parallel $\GG_2$ structure is linearly unstable and therefore  dynamically unstable. As 3-Sasakian 7-manifolds $M$ necessarily have $b^3(M)=0$, this class of examples admitting nearly parallel $\GG_2$-structures is particularly interesting for  Ricci flow in light of this result.

In this article, when discussing stability we will always be referring to dynamical stability.

%In our context we write down an ansatz for G$_{2}$-Structures originating from 3-Sasakian geometry which contains the NP G$_{2}$-Structures of which we are interested. As NP G$_{2}$-Structures simultaneously define special soliton solutions of the Laplacian flow, Laplacian coflow, and Ricci flow, these structures will never be fixed points of the flows as they will scale the structures. Instead, we call a NP G$_{2}$-Structure stable under the Laplacian flow or coflow if any G$_{2}$-Structure within our ansatz flows to the NP G$_{2}$-Structure up to overall scale. This notion of stability has implications for dynamical stability more generally. In particular, if a NP G$_{2}$-Structure is stable in our sense that is suggestive that it may be dynamically stable, but we would be able to infer if it is not stable in our sense then it cannot be dynamically stable. \\

\subsection{Main results} On any 3-Sasakian 7-manifold we introduce two disjoint 3-parameter families of coclosed $\GG_2$-structures defined in terms of the 3-Sasakian structure.  These families of $\GG_2$-structures each include exactly one of the natural nearly parallel $\GG_2$-structures we discussed above (and their rescalings).  We refer the reader to \S2 for details.  %We let $\varphi^{ts}$ and $\varphi^{np}$ denote these nearly parallel $\GG_2$-structures, which induce the 3-Sasakian and squashed Einstein metrics $g^{ts}$ and $g^{np}$ respectively and also let $\psi^{ts}$ and $\psi^{np}$ denote their dual 4-forms.   

Our main results concern the behaviour of the Laplacian coflow, the Laplacian flow and the Ricci flow for these families of coclosed $\GG_2$-structures and their induced metrics, which we show are preserved by the flows.  (Note, in particular, that the Laplacian flow is shown to preserve the coclosed condition in this setting.) 

Our most significant result is for the Laplacian coflow \eqref{eq:Lap.coflow}.

\begin{theorem}\label{thm:coflow}
The Laplacian coflow starting at any initial coclosed $\GG_2$-structure in either of our families converges, after rescaling, to the nearly parallel $\GG_2$-structure in that family.  In particular, the nearly parallel $\GG_2$-structures are both stable within their families.
\end{theorem}

Comparing the Laplacian flow \eqref{eq:Lap.flow} and Laplacian coflow \eqref{eq:Lap.coflow}, one might naively expect them to have similar behaviour as their velocities are Hodge dual.  However, in our setting, we have the following, which contrasts sharply with our Laplacian coflow result.

\begin{theorem}\label{thm:flow} 
Both nearly parallel $\GG_2$-structures are unstable sources within their families under the rescaled Laplacian flow, so coclosed $\GG_2$-structure in our families which are not nearly parallel cannot flow to either of them.  
\end{theorem}

Finally, for the Ricci flow \eqref{eq:Ric.flow}, we have the following, which differs again from our previous two results.

\begin{theorem}\label{thm:Ric} Along the rescaled Ricci flow for our families of metrics,
the 3-Sasakian metric is stable, whereas the squashed Einstein metric is a saddle point and so unstable. 
\end{theorem}

\noindent This result again shows that, whilst the Ricci flow and the induced flow of metrics \eqref{eq:Ric.flow.psi} from the Laplacian coflow are closely related, their behaviour can be markedly different.

\subsection{Summary} 
We begin in \S2 by discussing background on 3-Sasakian geometry, the nearly parallel G$_{2}$-Structures determined by these geometries, and our geometric flow ansatz.  We then study the behavior of the Laplacian coflow in \S3, the Laplacian flow in \S4, and the Ricci flow in \S5.  To do this, we reduce the study of each rescaled flow to the analysis of a nonlinear ODE system for two functions.    
%In the appendix, we explicitly solve the Laplacian coflow, Laplacian flow, and Ricci flow when a simplification of our general ansatz is valid.

\section{\texorpdfstring{$\mathrm{G}_2$}{G2}-structures on 3-Sasakian 7-manifolds}

In this section we recall some of the basics of 3-Sasakian geometry in 7 dimensions and outline its relationship to $\mathrm{G}_2$ geometry.  Further details on 3-Sasakian geometry can be found in \cite{BG,BGBook}.  For information about $\mathrm{G}_2$-structures, we refer the reader to \cite{JoyceBook} or \cite[pp.~3--50]{KLL}.

\subsection{3-Sasakian 7-manifolds}
We first recall the definition of a 3-Sasakian 7-manifold.

\begin{definition}\label{dfn:3Sak}
A complete Riemannian 7-manifold $(M^7,g_M)$ is \emph{3-Sasakian} if it has an orthonormal triple of Killing fields $\{E_1,E_2,E_3\}$ satisfying $[E_i,E_j]=2E_k$ for a cyclic permutation $(i,j,k)$ of $(1,2,3)$, such that each $E_i$ defines a Sasakian structure on $(M,g_M)$.
\end{definition}

If $(M,g_M)$ is 3-Sasakian then $g_M$ is Einstein with positive scalar curvature equal to 42 (so $M$ is compact) and there is a locally free action of $\mathrm{SU}(2)$ on $M$ whose leaf space $N$ is a 4-dimensional orbifold.  Moreover, there is a canonical metric $g_N$ on $N$, which is anti-self-dual Einstein with positive scalar curvature equal to 48, such that $(M, g_{M})$ and $(N, g_{N})$ are related by an orbifold Riemannian submersion: 
\begin{equation}\label{eq:M.Z.pi}
\pi: M \rightarrow N.
\end{equation}

\begin{remark}
The simplest example of a 3-Sasakian 7-manifold is the 7-sphere with its constant curvature $1$ metric. In this setting, \eqref{eq:M.Z.pi} just becomes the usual Hopf fibration with $M=S^7$ and $N=S^4$, and $N=S^4$ has its constant curvature $4$ metric.
\end{remark}

The Levi-Civita connection of $(N,g_N)$ lifts to a connection on the bundle \eqref{eq:M.Z.pi}, and so may be viewed as an $\mathfrak{su}(2)$-valued 1-form $\eta$ on $M$, which can be written as
\begin{equation}\label{eq:eta}
\eta = \sum_{i=1}^3\eta_{i} \otimes T_{i},
\end{equation}
where $\eta_1,\eta_2,\eta_3$ are 1-forms on $M$ and $\{T_1,T_2,T_3\}$ is a basis for $\mathfrak{su}(2)$ satisfying $[T_i,T_j]=2T_k$ for cyclic permutations $(i,j,k)$ of $(1,2,3)$. The curvature of $\eta$ is then an $\mathfrak{su}$(2)-valued 2-form $\omega$ which may be written as 
\begin{equation}\label{eq:omega}
\omega = -2\sum_{i=1}^3\omega_{i} \otimes T_{i}
\end{equation}
for 2-forms $\omega_1,\omega_2,\omega_3$ on $M$ which are, in fact, pullbacks of orthogonal self-dual 2-forms on $N$ since $g_N$ is anti-self-dual Einstein.  (The factor of $2$ and sign are chosen for convenience.)  Moreover, we have that the forms $\omega_1,\omega_2,\omega_3$ are normalized such that
\begin{equation}\label{eq:normalization}
    \omega_i\wedge\omega_j=2\delta_{ij}\pi^*\mathrm{vol}_N.
\end{equation}
For later use, we record the following equations satisfied by $\eta$ and $\omega$ where, in each case, $(i,j,k)$ are taken to be a cyclic permutation of $(1,2,3)$:
\begin{align}
    \mathrm{d}\eta_i&=-2\eta_j\wedge\eta_k-2\omega_i,\label{eq:deta}\\
    \mathrm{d}\omega_i&=-2\eta_j\wedge\omega_k+2\eta_k\wedge\omega_j.\label{eq:domega}
\end{align}

\noindent The 3-Sasakian metric $g_M$ on $M$ may then be given in terms of the $\eta_{i}$ and $g_{N}$ as follows:
\begin{equation}
    g_M=\eta_1^2+\eta_2^2+\eta_3^2+\pi^*g_N,
\end{equation}

We can scale $g_{M}$ by any positive constant $c$ and then $c^2g_{M}$ will still be Einstein with positive scalar curvature. We may also observe the following well-known fact.

\begin{lemma}
The metric 
\begin{equation}
    \tilde{g}_M=\frac{1}{5}(\eta_1^2+\eta_2^2+\eta_3^2)+\pi^*g_N
\end{equation}
is Einstein with positive scalar curvature and is known as the squashed Einstein metric on the 3-Sasakian $M^7$ \cite{FKMS, GS}.
\end{lemma}

\begin{remark}\label{rmk:holonomy}
The metric cone on $(M,g_M)$ has holonomy contained in $\mathrm{Sp}(2)$, whereas the metric cone on $(M,\tilde{g}_M)$ (once one scales $\tilde{g}_M$ appropriately) has holonomy $\mathrm{Spin}(7)$.  In the first case, the metric cone has the full holonomy $\mathrm{Sp}(2)$ if it is not flat.
\end{remark}

\subsection{Natural \texorpdfstring{$\mathrm{G}_2$}{G2}-structures}
We recall that a $\mathrm{G}_2$-structure on a 7-manifold is determined by a 3-form $\varphi$ on the manifold satisfying a certain nondegeneracy condition.  Such a 3-form determines a metric $g_\varphi$ and volume form $\mathrm{vol}_{\varphi}$, and hence a  dual 4-form $\psi=*_{\varphi}\varphi$, where $*_{\varphi}$ is the Hodge star determined by $\varphi$.

Given the data in \eqref{eq:M.Z.pi}, \eqref{eq:eta} and \eqref{eq:omega} above, we may now write down a natural family of $\mathrm{G}_2$-structures on a 3-Sasakian 7-manifold $(M^7,g_M)$ as follows.

\begin{lemma}\label{lem:ansatz}
Given $a_1,a_2,a_3,c>0$ and $\epsilon\in\{\pm1\}$, if we let $\mathbf{a}=(a_1,a_2,a_3)$ then the $3$-form 
\begin{equation}\label{eq:varphi.general}
    \varphi_{\mathbf{a},c,\epsilon}=\epsilon a_1a_2a_3\eta_{1}\wedge\eta_2\wedge\eta_3-c^2(a_1\eta_1\wedge\omega_1+a_2\eta_2\wedge\omega_2+\epsilon a_3\eta_3\wedge\omega_3)
\end{equation}
defines a $\mathrm{G}_2$-structure on $M$.  Moreover, this $\mathrm{G}_2$-structure induces the following metric, volume form and dual $4$-form:
\begin{align}
    g_{\mathbf{a},c}&=a_1^2\eta_1^2+a_2^2\eta_2^2+a_3^2\eta_3^2+c^2\pi^*g_N;\label{eq:metric.general}\\
   \mathrm{vol}_{\mathbf{a},c,\epsilon}&=\epsilon a_1a_2a_3c^4\eta_1\wedge\eta_2\wedge\eta_3\wedge\pi^*\mathrm{vol}_N;\\
       \psi_{\mathbf{a},c,\epsilon}&=c^4\pi^*\mathrm{vol}_N-c^2(\epsilon a_2a_3\eta_2\wedge\eta_3\wedge\omega_1+\epsilon a_3a_1\eta_3\wedge\eta_1\wedge\omega_2+a_1a_2\eta_1\wedge\eta_2\wedge\omega_3).\label{eq:psi.general}
\end{align}
Note that $g_{\mathbf{a},c}$ is independent of $\epsilon$.
\end{lemma}

\noindent This result is an elementary consequence of the fact that $\omega_1,\omega_2,\omega_3$ are the pullbacks of self-dual 2-forms on $N$ satisfying \eqref{eq:normalization}.

\begin{remark} Initially, one may allow for $a_1,a_2,a_3\in\mathbb{R}\setminus\{0\}$.  However, $\varphi$ and $-\varphi$ are the same $\mathrm{G}_2$-structure up to a change of orientation.  Moreover, there are only two possibilities: either $a_1,a_2,a_3$ all have the same sign, or just two have the same sign.  Therefore, we can take $a_1,a_2,a_3$ to be all positive and use $\epsilon$ to account for the two choices.
\end{remark}

We now compute the exterior derivatives of $\varphi_{\mathbf{a},c,\epsilon}$ and $\psi_{\mathbf{a},c,\epsilon}$, which together encode all of the information about the torsion of the $\GG_2$-structure.

\begin{lemma}\label{lem:torsion}
Let $\varphi_{\mathbf{a},c,\epsilon}$ and $\psi_{\mathbf{a},c,\epsilon}$ be as in Lemma \ref{lem:ansatz}.  Then:
\begin{align*}
    \mathrm{d}\varphi_{\mathbf{a},c,\epsilon}&=4c^2(a_1+a_2+\epsilon a_3)\pi^*\mathrm{vol}_N\\
&\quad    -2(\epsilon a_1a_2a_3-c^2a_1+c^2a_2+\epsilon c^2a_3)\eta_{2}\wedge\eta_3\wedge\omega_1\\
    & \quad -2(\epsilon a_1a_2a_3+c^2a_1-c^2a_2+\epsilon c^2a_3)\eta_3\wedge\eta_1\wedge\omega_2\\
    &\quad-2(\epsilon a_1a_2a_3+c^2a_1+c^2a_2-\epsilon c^2a_3)\eta_1\wedge\eta_2\wedge\omega_3; \\
 \mathrm{d}\psi_{\mathbf{a},c,\epsilon}&=0.   
\end{align*}
\end{lemma}
\noindent This result follows quickly from \eqref{eq:deta} and \eqref{eq:domega}.  Notice in particular that the $\GG_2$-structures are all \emph{coclosed}.  

\begin{remark}\label{rmk:special}
We note the following special cases of our family of $\mathrm{G}_2$-structures.
\begin{itemize}
\item We can always make an overall rescaling so that $c=1$.  (However, we shall see that we will require the freedom to vary the scale $c$ along our flows.)  
\item Taking $a_1=a_2=a_3=c=1$ and $\epsilon=1$ gives a coclosed $\mathrm{G}_2$-structure inducing the 3-Sasakian  metric. %It has been shown to encode the data for all relevant nearly-parallel $\mathrm{G}_2$-structures in a fundamental way \cite{AF}. 
It has been referred to as the ``canonical'' $\mathrm{G}_2$-structure on a 3-Sasakian 7-manifold (see e.g.~\cite{AF}). %, although we will not use this terminology.   
\item Taking $a_1=a_2=a_3=a$ and $c=1$ gives the family of $\mathrm{G}_2$-structures considered in \cite{LO}.  The subfamily where $\epsilon=1$ was also studied earlier in \cite{FKMS}. 
\end{itemize}
\end{remark}

\subsection{Nearly parallel G\textsubscript{2}-structures} We recall the definition of the distinguished class of $\GG_2$-structures that will be the focus of this paper.

\begin{definition}
A $\GG_2$-structure on a 7-manifold $M$ defined by a 3-form $\varphi$ with dual 4-form $\psi$ is \emph{nearly parallel} if
$$\mathrm{d}\varphi=\lambda \psi$$
for some non-zero constant $\lambda$. %Most fundamentally, a nearly parallel $\mathrm{G}_2$-structure is of the Fernandez-Gray class such that only $\tau_{0}$ does not vanish. This $\tau_{0}$ is a priori a nontrivial function, however the fact that this definition implies $\psi$ is closed coupled with a short representation theory argument implies that $\tau_{0}$ is actually the constant which we take to be $\lambda$.  
(A priori $\lambda$ could be a function on $M$, but a short argument using $\d\psi=0$ and some representation theory shows that it must in fact be constant.)

A nearly parallel $\GG_2$-structure $\varphi$ induces an Einstein metric $g_{\varphi}$ with positive scalar curvature.  If $\lambda$ is chosen so that the scalar curvature of $g_{\varphi}$ is 42, then the cone metric $\mathrm{d}r^2+r^2g_{\varphi}$ on $\mathbb{R}^+\times M$ is Ricci-flat and has holonomy contained in $\mathrm{Spin}(7)$, and $\varphi$ is  \emph{strictly nearly parallel} if the holonomy of this cone metric is $\mathrm{Spin}(7)$.  (One should compare this to Remark \ref{rmk:holonomy}.)
\end{definition}

We now record the following facts, which follow immediately from \eqref{eq:deta} and \eqref{eq:domega}, that show that our family of $\GG_2$-structures contains two nearly parallel $\GG_2$-structures (up to scale).

\begin{lemma}\label{lem:np} Take $a_1=a_2=a_3=a$ in $\varphi_{\mathbf{a},c,\epsilon}$.
\begin{itemize}
    \item If $a=\frac{1}{\sqrt{5}}c$ and $\epsilon=1$, then the resulting $\mathrm{G}_2$-structure, which we may write  $c^3\varphi^{np}$ with $\varphi^{np}$ independent of $c$, is (strictly) nearly parallel and its induced metric is $c^2\tilde{g}_M$.
    \item If $a=c$ and $\epsilon=-1$, then the resulting $\mathrm{G}_2$-structure, which we may write  $c^3\varphi^{ts}$ where $\varphi^{ts}$ is independent of $c$, is nearly parallel and its induced metric is $c^2g_M$.
\end{itemize}
\end{lemma}
\noindent Hence, within each branch (determined by $\epsilon\in\{\pm1\}$) of our family of $\GG_2$-structures, there is one natural critical point (up to scale) for our geometric flows.  For $\epsilon=1$, this is the strictly nearly parallel $\GG_2$-structure $\varphi^{np}$ inducing the squashed Einstein metric $\tilde{g}_M$ on $M$, and for $\epsilon=-1$ this is the nearly parallel $\GG_2$-structure $\varphi^{ts}$ (where ``ts" stands for 3-Sasakian) inducing the 3-Sasakian metric $g_M$.

\begin{remark}
If we take $a=\frac{1}{\sqrt{5}}c$ and $\epsilon=-1$ in Lemma \ref{lem:np} then we obtain a coclosed $\GG_2$-structure which  induces the Einstein metric $c^2\tilde{g}_M$, but is not  nearly parallel.  The same occurs when we  $a=c$ and $\epsilon=1$, but now for the Einstein metric $c^2g_M$: this gives a multiple of the ``canonical'' $\GG_2$-structure  we saw earlier (cf.~Remark \ref{rmk:special}). 
\end{remark}

\begin{remark}
    It is worth noting that, by Lemma \ref{lem:np} and \cite[Examples 1.14 and 1.15]{CN}, the $\GG_2$-structures defined by $\varphi_{\mathbf{a},c,+1}$ and $\varphi_{\mathbf{a},c,-1}$ cannot be homotopic through $\GG_2$-structures.
\end{remark}

\subsection{The ansatz}\label{ss:ansatz} Motivated by Lemma \ref{lem:np}, we will take our ansatz to be a special case of that of Lemma \ref{lem:ansatz} where
\begin{equation}\label{eq:ansatz}
    a_1=a_2=a(t),\qquad a_3=b(t)\quad\text{and}\quad c=c(t),
\end{equation}
for positive time-dependent functions $a(t),b(t),c(t)$. These then define 1-parameter families of $\GG_2$ 3-forms $\varphi_{\epsilon}(t)$ depending on $t$, with induced metric $g(t)$, volume form $\mathrm{vol}_{\epsilon}(t)$ and dual 4-form $\psi_{\epsilon}(t)$ as follows: 
\begin{align}
      \varphi_{\epsilon}(t)&=\epsilon a(t)^2b(t)\eta_{1}\wedge\eta_2\wedge\eta_3-a(t)c(t)^2(\eta_1\wedge\omega_1+\eta_2\wedge\omega_2)-\epsilon b(t)c(t)^2\eta_3\wedge\omega_3;\label{eq:varphi.ansatz}\\
    g(t)&=a(t)^2(\eta_1^2+\eta_2^2)+b(t)^2\eta_3^2+c(t)^2\pi^*g_N;\label{eq:metric.ansatz}\\
   \mathrm{vol}_{\epsilon}(t)&=\epsilon a(t)^2b(t)c(t)^4\eta_1\wedge\eta_2\wedge\eta_3\wedge\pi^*\mathrm{vol}_N;\\
       \psi_{\epsilon}(t)&=c(t)^4\pi^*\mathrm{vol}_N-\epsilon a(t)b(t)c(t)^2(\eta_2\wedge\eta_3\wedge\omega_1+\eta_3\wedge\eta_1\wedge\omega_2)-a(t)^2c(t)^2\eta_1\wedge\eta_2\wedge\omega_3.\label{eq:psi.ansatz}
\end{align}
We include the subscript $\epsilon$ to emphasise the choice of branch given by $\epsilon\in\{\pm 1\}$, as we shall see different behaviour for distinct choices of $\epsilon$, but drop the subscript for $g(t)$ since it is independent of $\epsilon$.  We shall make the restriction in \eqref{eq:ansatz} henceforth in this article.

\begin{remark}
The reader way wonder why we do not simply choose $a=b$ in \eqref{eq:ansatz} given that this holds for the nearly parallel $\GG_2$-structures in Lemma \ref{lem:np}.  We shall see that the simpler ansatz when $a=b$ is not necessarily preserved along the geometric flows we consider, and so we broaden our study to consider the larger class of 1-parameter families of $\GG_2$-structures given by the condition \eqref{eq:ansatz}.  One could also consider curves in the full family of $\GG_2$-structures in Lemma \ref{lem:ansatz}, but this would be much more challenging and we already exhibit interesting behaviour within the framework provided by \eqref{eq:ansatz}. 
\end{remark}

  For the ansatz, we have the following simplification and slight extension of Lemma \ref{lem:torsion}.

\begin{lemma}\label{lem:torsion2}
Let $\varphi_{\epsilon}=\varphi_{\epsilon}(t)$ and $\psi_{\epsilon}=\psi_{\epsilon}(t)$ be given by Lemma \ref{lem:ansatz} with the conditions in \eqref{eq:ansatz}.  Then:
\begin{align*}
    \mathrm{d}\varphi_{\epsilon}&=4c^2(2a+\epsilon b)\pi^*\mathrm{vol}_N -2\epsilon b(a^2+c^2)\eta_{2}\wedge\eta_3\wedge\omega_1\\
    & \quad -2\epsilon b(a^2+c^2)\eta_3\wedge\eta_1\wedge\omega_2-2\epsilon(a^2b+2\epsilon ac^2-bc^2)\eta_1\wedge\eta_2\wedge\omega_3; \\
 \mathrm{d}\psi_{\epsilon}&=0.   
\end{align*}
Moreover, we may write $\mathrm{d}\varphi_{\epsilon}=\tau_0\psi_{\epsilon}+*\tau_3$ where
\begin{align*}
   \tau_0&=\frac{4}{7a^2c^2}\left(4a(a^2+c^2)+\epsilon b(2a^2-c^2) \right)%\\
 %  \ast\tau_3&= \Big(\frac{20c^2}{7} (2a + \epsilon b) + \frac{4c^4}{7a^2}(\epsilon b -4a)\Big) \pi^*\mathrm{vol}_Z \\
   %&+\Big(\frac{2\epsilon b}{7}(a^2+c^2)+ \frac{4b^2}{7a}(2a^2-c^2) \Big)\eta_{2}\wedge\eta_{3}\wedge\omega_{1} \\
   %& + \Big(\frac{2\epsilon b}{7}(a^2+c^2) + \frac{4b^2}{7a}(2a^2-c^2) \Big)\eta_{3}\wedge\eta_{1}\wedge\omega_{2}\\
   %& +\Big(\frac{2\epsilon b}{7} (5c^2-3a^2) +\frac{4a}{7} (4a^2-3c^2)\Big)\eta_{1}\wedge\eta_{2}\wedge\omega_{3}
\end{align*}
and $\tau_3\wedge\varphi_{\epsilon}=0=\tau_3\wedge\psi_{\epsilon}$.
\end{lemma}

\begin{proof}
The formulas for $\d\varphi_{\epsilon}$ and $\d\psi_{\epsilon}$ are immediate from Lemma \ref{lem:torsion}.  We then compute that
\begin{align*}
    \mathrm{d}\varphi_{\epsilon}\wedge\varphi_{\epsilon}&=4\big(\epsilon a^2bc^2(2a+\!\epsilon b)+2\epsilon abc^2(a^2+c^2)+bc^2(a^2b+2\epsilon ac^2-\!bc^2)\big)\eta_1\wedge\eta_2\wedge\eta_3\wedge\pi^*\mathrm{vol}_N\\
    &=4\epsilon bc^2\big(4a(a^2+c^2)+\epsilon b(2a^2-c^2)\big)\eta_1\wedge\eta_2\wedge\eta_3\wedge\pi^*\mathrm{vol}_N.
\end{align*}
 The formula for $\tau_0$ follows. %To find $\ast\tau_{3}$ we subtract off $\tau_{0}\psi_{\epsilon}$ from $\d\varphi_{\epsilon}$ to obtain:
%\begin{align*}
%\ast \tau_{3} &= (4c^2(2a+\epsilon b) - c^4\tau_{0})\pi^*\mathrm{vol}_Z -(2\epsilon b(a^2+c^2) - ab\epsilon c^2\tau_{0}) \eta_{2}\wedge \eta_{3}\wedge\omega_{1} \\
%& -(2\epsilon b(a^2+c^2) - ab\epsilon c^2\tau_{0}) \eta_{3}\wedge \eta_{1}\wedge\omega_{2} - (2\epsilon (a^2b + 2\epsilon ac^2 - bc^2) - a^2c^2\tau_{0})\eta_{1}\wedge\eta_{2}\wedge\omega_{3}
%\end{align*}
%The result for $\ast\tau_{3}$ follows after substituting in the formula for $\tau_0$.
\end{proof}

\begin{remark}
Lemma \ref{lem:torsion2} shows that, regardless of the choice of $\epsilon\in\{\pm 1\}$, we can always choose initial conditions for our flows of $\GG_2$-structures such that $\tau_0=0$ (and necessarily $\tau_3\neq 0$), even though we are trying to flow to nearly parallel $\GG_2$-structures, which must have $\tau_0\neq 0$ and $\tau_3=0$.
\end{remark}

\section{Laplacian coflow}

We start by studying the Laplacian coflow, which is arguably the natural flow for our ansatz of coclosed $\mathrm{G}_2$-structures since it manifestly preserves the coclosed condition.  We recall that this flow, if it is well-posed and stays within the ansatz, is given by
\begin{equation}\label{eq:coflow.ansatz}
    \frac{\partial}{\partial t}\psi_{\epsilon}(t)=\Delta_{\psi_{\epsilon}(t)}\psi_{\epsilon}(t)=\mathrm{d}\mathrm{d}^*_{\psi_{\epsilon}(t)}\psi_{\epsilon}(t),
\end{equation}
for the closed 4-forms $\psi_{\epsilon}(t)$ in \eqref{eq:psi.ansatz}.

\subsection{The flow equations} Since we have that
\[
\mathrm{d}\mathrm{d}^*_{\psi_{\epsilon}}\psi_{\epsilon}=\mathrm{d}*\mathrm{d}\varphi_{\epsilon},
\]
it is straightforward to compute the right-hand side of \eqref{eq:coflow.ansatz} from Lemma \ref{lem:torsion2} as follows.

\begin{lemma}\label{lem:Lap.psi}
The Hodge Laplacian of $\psi_{\epsilon}$ in \eqref{eq:psi.ansatz} is given by:
\begin{equation}\label{eq:Lap.psi}
\begin{split}
\Delta_{\psi_{\epsilon}}\psi_{\epsilon} &= 8\left(2a^2 +b^2 +2c^2 + \frac{2\epsilon bc^2}{a} - \frac{b^2c^2}{a^2}\right) \pi^{\ast} \textnormal{vol}_{N} \\
&\qquad - 4\Big(b^2+\frac{4\epsilon a^3 b}{c^2} + \frac{2a^2b^2}{c^2} + \frac{2\epsilon bc^2}{a} - \frac{b^2c^2}{a^2} \Big) (\eta_{2}\wedge\eta_3\wedge\omega_{1}+ \eta_{3}\wedge\eta_1\wedge\omega_{2} ) \\
&\qquad - 4\Big(2a^2 -b^2 + 2c^2+\frac{4\epsilon a^3 b}{c^2} + \frac{2a^2b^2}{c^2} -\frac{2\epsilon bc^2}{a} + \frac{b^2c^2}{a^2} \Big) \eta_{1}\wedge\eta_2\wedge\omega_{3}
\end{split}
\end{equation}
In particular, \eqref{eq:Lap.psi} is in  the same form as \eqref{eq:psi.ansatz} and so the Laplacian coflow  \eqref{eq:coflow.ansatz} is well-defined.
\end{lemma}

Given this result and \eqref{eq:psi.ansatz}, we may write down the Laplacian coflow \eqref{eq:coflow.ansatz} as the following system of ordinary differential equations for the coefficient functions $a(t)$, $b(t)$, $c(t)$:
\begin{align*}
\frac{\mathrm{d}}{\mathrm{d} t}(c^4) & = 8\left(2a^2 +b^2+ 2c^2 +\frac{2\epsilon bc^2}{a} - \frac{b^2c^2}{a^2}\right);  \\
\frac{\mathrm{d} }{\mathrm{d} t}(a^2c^2)& = 4\left(2a^2-b^2+2c^2+\frac{4\epsilon a^3b}{c^2}+\frac{2a^2b^2}{c^2}-\frac{2\epsilon bc^2}{a}   + \frac{b^2c^2}{a^2}\right); \\  
\frac{\mathrm{d}}{\mathrm{d} t}(abc^2) &= 4\left(\epsilon b^2+\frac{4a^3b}{c^2} + \frac{2\epsilon a^2b^2}{c^2} + \frac{2 bc^2}{a} -  \frac{\epsilon b^2c^2}{a^2}\right). 
\end{align*}
We can simplify the analysis of these equations by introducing new variables as follows.

\begin{lemma}\label{lem:coflow.eqns}
Define
\[
X=\frac{a^2}{c^2} \quad\text{and}\quad Y=\frac{ab}{c^2}
\]
and introduce a new variable $s$ by 
\[
\frac{\mathrm{d} s}{\mathrm{d} t}=\frac{1}{c^2}.
\]
If we let $\dot{X}=\frac{\mathrm{d} X}{\mathrm{d} s}$ and  $\dot{Y}=\frac{\mathrm{d} Y}{\mathrm{d} s}$, then the Laplacian coflow equations for \eqref{eq:psi.ansatz} imply that
\begin{align}
    \dot{X} &= \frac{4}{X^2}%\left((4\epsilon X^3Y-4\epsilon X^2Y-2\epsilon XY -4X^4-2X^3 +XY^2+2X^2+Y^2\right)
    \left((X+1)Y^2+2\epsilon (2X^2-2X-1)XY-2X^2(2X-1)(X+1)\right);\label{eq:coflow1}\\
    \dot{Y} &=\frac{4Y}{X^2}%\left(2\epsilon X^2Y-3\epsilon XY-2XY^2-4X^2+2Y^2+2X-\epsilon Y\right).
    \left(2(1-X)Y^2+\epsilon (2X^2-3X-1)Y+2X(1-2X)\right).\label{eq:coflow2}
\end{align}
\end{lemma}

\noindent We note that $X$ and $Y$ are scale-invariant quantities and that solutions to \eqref{eq:coflow1}--\eqref{eq:coflow2}  give the solutions to the Laplacian coflow \eqref{eq:coflow.ansatz} up to rescaling.

\subsection{Critical points and dynamics} To understand the dynamics of the flow  \eqref{eq:coflow1}--\eqref{eq:coflow2}, we need to study its critical points.  Some straightforward calculations show the following.

\begin{lemma}\label{lem:coflow.critpts}
The only critical points for $X,Y>0$ to the system \eqref{eq:coflow1}--\eqref{eq:coflow2} are:
\begin{equation}\label{eq:coflow.crit.1}
X=Y=\frac{1}{5} \quad\text{and}\quad \epsilon=1
\end{equation}
and
\begin{equation}\label{eq:coflow.crit.2}
X=Y=1\quad\text{and}\quad \epsilon=-1.
\end{equation}
Moreover, if $\epsilon=1$ the condition $X=Y$ is preserved, but if $\epsilon=-1$ the condition $X=Y$ is not preserved except when $X=Y=1$.
\end{lemma}

 \begin{remark}\label{rmk:psi.np}  By Lemma \ref{lem:np}, the critical points \eqref{eq:coflow.crit.1}--\eqref{eq:coflow.crit.2} correspond to the 4-forms $\psi^{np}$ and $\psi^{ts}$ dual to the nearly parallel $\GG_2$-structures $\varphi^{np}$ and $\varphi^{ts}$ respectively.  Thus, Lemma \ref{lem:coflow.critpts} shows that the only critical points for \eqref{eq:coflow.ansatz} up to rescaling are $\psi^{np}$ and $\psi^{ts}$. 
\end{remark}

Before considering the general ansatz, we note that if we set $X=Y$ and $\epsilon=1$ in \eqref{eq:coflow1}--\eqref{eq:coflow2} then we obtain:
\begin{equation*}
    \dot{X}=4(1-5X).
\end{equation*}
Hence, $\dot{X}$ is positive for $X<1/5$ and negative for $X>1/5$,
which clearly shows the stability along the line $X=Y$ of the critical point \eqref{eq:coflow.crit.1}.  Thus $\psi^{np}$ is stable within the restricted ansatz \eqref{eq:ansatz} with $a=b$.

\begin{remark} \label{rem:CoflowSimple} 
Lemma \ref{lem:coflow.critpts} shows  that the coclosed $\GG_2$-structure with $\epsilon=1$ inducing the 3-Sasakian metric (up to scale), as well as the one with $\epsilon=-1$ inducing the squashed Einstein metric, have no significance for the Laplacian coflow.  It also shows that we need to use the ansatz \eqref{eq:ansatz} with $a$ and $b$ distinct (i.e.~allowing $X\neq Y$) to understand the Laplacian coflow for $\epsilon=-1$.
\end{remark}

We provide dynamic plots of the equations \eqref{eq:coflow1}--\eqref{eq:coflow2} in Figure \ref{fig:coflow.plots} for $\epsilon=\pm 1$.  In the plots, the curves $\gamma_X$ and $\gamma_Y$ across which $\dot{X}$ and $\dot{Y}$ change sign respectively are also shown, along with the line $X=Y$.

\begin{figure}[ht]\caption{Dynamic plots for Laplacian coflow for $\epsilon=1$ and $\epsilon=-1$}
    \centering
    \includegraphics[width=0.45\textwidth]{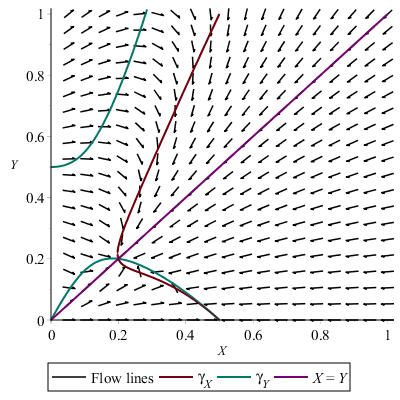}\qquad
     \includegraphics[width=0.45\textwidth]{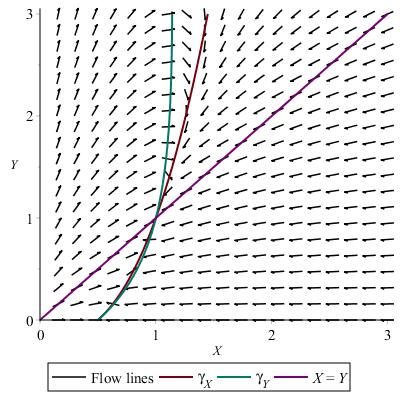}
    \label{fig:coflow.plots}
\end{figure}

Figure \ref{fig:coflow.plots} indicates that the critical points \eqref{eq:coflow.crit.1} and \eqref{eq:coflow.crit.2}, which correspond to $\psi^{np}$ and $\psi^{ts}$ as in Remark \ref{rmk:psi.np}, are both \emph{stable}.  We now show that this is indeed true.

\begin{proposition}
The $4$-forms $\psi^{np}$ and $\psi^{ts}$ dual to the nearly parallel $\GG_2$-structures $\varphi^{np}$ and $\varphi^{ts}$ are stable sinks under the Laplacian coflow \eqref{eq:coflow.ansatz}, after rescaling.
\end{proposition}

\begin{proof}
We  study the linearization of the flow equations \eqref{eq:coflow1}--\eqref{eq:coflow2} at the critical points  \eqref{eq:coflow.crit.1} and \eqref{eq:coflow.crit.2} to determine their stability. 

At $X=Y=1/5$ and $\epsilon=1$, the linearized equations are
\begin{align*}
\dot{X} = -\frac{76}{5} X - \frac{24}{5}Y\quad\text{and}\quad
\dot{Y} = -\frac{12}{5}X -\frac{88}{5}Y.
\end{align*}
(Note that $X=Y$ is preserved by the above system as expected.)
The associated $2\times 2$ matrix of  coefficients of $X,Y$ in the above equations has two negative eigenvalues ($-64/5$ and $-20$) and so \eqref{eq:coflow.crit.1} is a stable critical point. 

Similarly, at $X=Y=1$ and $\epsilon=-1$, the linearized equations are 
\begin{align*}
\dot{X} = -76X + 24Y \quad\text{and}\quad
 \dot{Y} = -36X + 8Y,
\end{align*}
noting that this time $X=Y$ is not preserved.  Here, the matrix one obtains again has two negative eigenvalues, which are $-4$ and $-64$, so the critical point \eqref{eq:coflow.crit.2} is stable.
\end{proof}

\begin{remark}\label{rmk:coflow.collpase}
We see from \eqref{eq:coflow1}--\eqref{eq:coflow2} that if we allow $X=0$ or $Y=0$ then there are additional critical points:
\[
(X,Y)=\left(\frac{1}{2},0\right)\quad\text{for $\epsilon=\pm 1$.}
\]
We can see these critical points  in Figure \ref{fig:coflow.plots}. We can also consider $(X,Y)=(0,0)$ to be a degenerate critical point, even though the equations \eqref{eq:coflow1}--\eqref{eq:coflow2} are not defined there. 
%, where it would appear to be unstable, as we can confirm by Theorem \ref{thm:coflow}.
%by computing the linearization of the flow at this point:  
%\begin{align*}
% \dot{X} = -24X-24\epsilon Y ,\quad
%\dot{Y} = 0,
%\end{align*}
%
%The $\pm$ depends on the value of $\epsilon$, however, in either case this critical point is clearly unstable. \\
%
We can understand these additional critical points geometrically as follows.

Recall the fibration \eqref{eq:M.Z.pi} of $M^7$ over $N^4$.  The point $(0,0)$ corresponds to sending the 3-dimensional fibres of \eqref{eq:M.Z.pi} to zero size (since $a=b=0$), and so $M$ has collapsed to $N$ (or a point).  In this setting, the 4-form $\psi_{\epsilon}$ reduces to simply the volume form of $N$ (or zero if the collapse is to a point).

If we instead view the fibres of \eqref{eq:M.Z.pi} as circle bundles over $S^2$ (where  $E_3$ is tangent to the circle direction in the notation of Definition \ref{dfn:3Sak}), then at $(1/2,0)$ the circle fibres have now collapsed (as $b=0$ in \eqref{eq:ansatz}).  Since $a\neq 0$ there $M^7$ has collapsed to a 6-manifold $Z$ which is a 2-sphere bundle over $N$.  This 6-manifold $Z$ is the \emph{twistor space} of $N$, and at this critical point it will be endowed with its \emph{nearly K\"ahler} metric $g_Z$, which is an Einstein metric on $Z$ with positive scalar curvature.  This Einstein metric $g_Z$ is  not K\"ahler (unlike the  standard choice of metric on the twistor space), but instead is related to $\GG_2$ geometry as the metric cone over a 6-dimensional nearly K\"ahler manifold will have holonomy $\GG_2$.
\end{remark}

\subsection{Long-time behaviour}  The plots in Figure \ref{fig:coflow.plots} suggest that, within our ansatz,  any initial condition  flows to the unique (up to scale) nearly parallel $\GG_2$-structure in the family.  We now show that this is indeed the case.  For the statement, as in Remark \ref{rmk:psi.np}, we denote by $\psi^{np}$ and $\psi^{ts}$ the duals of the nearly parallel $\GG_2$-structures $\varphi^{np}$ and $\varphi^{ts}$ defined in Lemma \ref{lem:np}, and recall that they induce the squashed Einstein metric  and 3-Sasakian  metric  respectively.

 %The 3-Sasakian nearly parallel G$_2$ structure $\varphi^{ts}$ is the only stable critical point for the Laplacian coflow starting at a G$_{2}$-Structure of the ansatz \eqref{eq:ansatz} with $\epsilon=-1$. Similarly, the squashed nearly parallel G$_{2}$-Structure $\varphi^{np}$ is the only stable critical point for the Laplacian coflow starting at a G$_{2}$-Structure of the same ansatz with $\epsilon=1$. As a result, the Laplacian coflow starting anywhere in the ansatz \eqref{eq:ansatz} takes the G$_{2}$-Structure to the appropriate nearly parallel G$_{2}$-Structure given the value of $\epsilon$.

\begin{theorem}\label{thm:coflow.2}
Let $\psi_{\epsilon}=\psi_\epsilon(0)$ be a closed $4$-form as in \eqref{eq:psi.ansatz} dual to a $\GG_2$-structure.  The solution to the Laplacian coflow \eqref{eq:coflow.ansatz} starting at $\psi_\epsilon$   converges, after rescaling, to $\psi^{np}$ if $\epsilon=1$ and to  $\psi^{ts}$ if $\epsilon=-1$, which are the only critical points of the rescaled flow.  In particular, the nearly parallel $\GG_2$-structures given by $\psi^{np}$ and $\psi^{ts}$ are stable for \eqref{eq:coflow.ansatz} after rescaling.
\end{theorem}

Theorem \ref{thm:coflow.2} gives Theorem \ref{thm:coflow} in the introduction. 
The proof of Theorem \ref{thm:coflow.2} is quite lengthy, so we break it up into several smaller results.

\subsubsection{Strategy}\label{sss:strategy} 
Our aim is to prove that, given any initial condition, the flow \eqref{eq:coflow1}--\eqref{eq:coflow2} for $(X(s),Y(s))$ exists for all $s$ and converges to the critical point with $X>0$ and $Y>0$.  From \eqref{eq:coflow1}--\eqref{eq:coflow2}, long time existence will be guaranteed as long as $X,Y$ remain bounded and $X$ remains bounded away from $0$.  Moreover,  periodic orbits are not possible as we know the Laplacian coflow is a gradient flow. Hence, long time existence will imply convergence to the critical point with $X,Y>0$ if we additionally know that $Y$ remains bounded away from $0$.   

To deduce the result for the Laplacian coflow we observe that the evolution equation for $c^2$, which determines the parameter $s$ by Lemma \ref{lem:coflow.eqns}, is:
\[
\frac{\d}{\d t}(c^2)=8\left(2X+\frac{Y^2}{X}+2+2\epsilon\frac{Y}{X}-\frac{Y^2}{X^2}\right),
\]
so $c^2$ cannot blow up in finite time because $X,Y$ are bounded and bounded away from zero.  In fact, $c^2$ grows at most linearly, so we can integrate $\frac{\d s}{\d t}=\frac{1}{c^2}$ to find $s$ as a function of the Laplacian coflow parameter $t$.

Throughout the proof, we denote the curves where $\dot{X}$ and $\dot{Y}$ change sign, respectively, by $\gamma_X$ and $\gamma_Y$ as in Figure \ref{fig:coflow.plots}.  

\subsubsection{Bounds on $X$} We start by looking at the behaviour of $X$.

\begin{lemma}\label{lem:X.bound}
The function $X(s)$ is bounded away from zero and can only diverge in the region to the left of $\gamma_X$.
\end{lemma}

\begin{proof}
We first see that $X$ is decreasing whenever it is to the right of $\gamma_X$ and so $X$ will remain bounded by its initial condition in this region.  We then see that $X$ is increasing in the region to the left of $\gamma_X$, which is the region containing $X=0$, and so $X$ is bounded away from zero in this region by its initial condition.  
\end{proof}

Since $\gamma_X$ meets $Y=X$ (at the critical point), the unbounded part of $\gamma_X$ lies above the line $Y=X$, i.e.~where $Y>X$.  Therefore, it suffices to show that $Y$ remains bounded in the region to the left of $\gamma_X$ where additionally $Y>X$ to deduce that $X$ is bounded everywhere.

\subsubsection{Bounds on $Y$: $\epsilon=1$} Given our earlier discussion, we now turn to showing that $Y$ remains bounded and bounded away from zero. We start with the case $\epsilon=1$.  

\begin{lemma}\label{lem:1.Y.bound}
When $\epsilon=1$, $Y(s)$ can only diverge in the region above the upper part of $\gamma_Y$ and can only tend to zero in the region above the lower part of $\gamma_Y$ but below the line $Y=X$.
\end{lemma}

\begin{proof}
In this case,  $\dot{Y}<0$ in the region above the lower part of $\gamma_Y$ and below or to the right of the upper part of $\gamma_Y$, so  $Y$ remains bounded by its initial value in this region.  Note that $\gamma_Y$ meets the line $X=0$ at $Y=0$ and $Y=1/2$.  Hence, in the same region we just considered, we are either to the left of $\gamma_X$ and so $\dot{X}>0$, which means  we cannot reach $(0,0)$, or we are to the right of $\gamma_X$ but also above $\gamma_Y$.  In this latter region, if we are above the line $Y=X$ we cannot cross it by Lemma \ref{lem:coflow.critpts}, and so $Y$ remains bounded away from zero here.   

We then notice that $\dot{Y}>0$  in the region bounded by the lower part of $\gamma_Y$, in which $Y$ is bounded, and so $Y$ is additionally bounded away from $0$ here.  %This completes the proof.
\end{proof}

We now have our crucial observation for the case $\epsilon=1$. 

\begin{lemma}\label{lem:1.Y.bound.2}
When $\epsilon=1$, $Y(s)$ is decreasing and hence bounded in the region where $Y>X$ and $X\geq 1$.
\end{lemma}

\begin{proof}
In this setting, we can rewrite \eqref{eq:coflow2} in the following manner:
\begin{equation*}
\dot{Y}=\frac{4Y}{X^2}\left(2(1-X)(Y-X)Y-XY-Y+2X(1-2X)\right).    
\end{equation*}
Hence, $\dot{Y}<0$ when $Y>X$ and $X\geq 1$.  Therefore, $Y$ will be bounded in this region.  
\end{proof}

Since the only part of the quadrant with $X,Y>0$ where $Y$ can become unbounded is where $Y>X$ by Lemma \ref{lem:1.Y.bound}, we deduce that $Y$ can only become unbounded, when $\epsilon=1$, if $X$ remains in $(0,1)$.  We now show that this is impossible.

\begin{proposition}\label{prop:1.Y.unbounded} For $\epsilon=1$, there are no solutions $(X(s),Y(s))$ above the upper part of $\gamma_Y$ with $X(s)\in(0,1)$  and $Y(s)$ unbounded.
\end{proposition}

\begin{proof}
Suppose not and that we have such a solution.  Note that there is a finite $\bar{Y}>0$ such that for all $X\in(0,1)$,
\begin{equation} \label{eq:coflowterms}
Y^2+ 2 (2X^2-2X-1)XY -2X^2(2X-1)(X+1) > 0\quad\text{and} \quad (2X^2-3X-1)Y+2X(1-2X)<0 
\end{equation}
for all $Y>\bar{Y}$.  Since $Y$ is increasing and unbounded in the region under consideration, we may assume that $Y(s)>\bar{Y}$.  Then, comparing \eqref{eq:coflowterms} and \eqref{eq:coflow1}--\eqref{eq:coflow2}, we see that
\begin{equation*}
\dot{X} \geq \frac{4}{X}Y^2\quad\text{and}\quad \dot{Y}\leq \frac{8(1-X)}{X^2}Y^3.
\end{equation*}
Hence, we see that
\[
\frac{\d Y}{\d X}\leq \frac{2(1-X)}{X}Y.
\]
Gr\"onwall's inequality then shows that there are constants $C_0,C_1$ depending only on the initial conditions so that $Y$ is bounded by $C_0X^2e^{-2X}+C_1$.  Since we assumed $X(s)\in(0,1)$, this is a contradiction.
\end{proof}

Our results so far show that, when $\epsilon=1$, both $X$ and $Y$ are bounded and that $X$ is bounded away from zero. 
To complete the proof of Theorem \ref{thm:coflow.2} in the $\epsilon=1$ case we therefore only need the following.

\begin{lemma}\label{lem:1.Y.bound.0}
When $\epsilon=1$, $Y(s)$ is bounded away from zero.
\end{lemma}

\begin{proof}
By Remark \ref{rmk:coflow.collpase} and Lemma \ref{lem:1.Y.bound},  we need only consider points near $(X,Y)=(1/2,0)$ above $\gamma_Y$.  Linearizing around the critical point $(1/2,0)$, so writing $X=\frac{1}{2}+\delta X_1$ and $Y=\delta Y_1$ for $\delta$ small, we find that the linear term in $\delta$ gives
\[
\dot{X_1}+\dot{Y_1} = -24(X_1+Y_1) \quad\text{and}\quad \dot{Y_1}=0.
\]
These equations are degenerate along the line $X_1+Y_1=0$, which is  tangent to $\gamma_Y$ at $(1/2,0)$. 
However, $\gamma_Y$ lies above this line for all points near $(1/2,0)$, so we may restrict to the region where $X_1+Y_1>0$.  Therefore, to leading order, $Y_1$ remains a non-zero constant and $X_1+Y_1$ decreases with an exponential rate.  Hence, $Y$ will be bounded away from $0$, as claimed.
\end{proof}

We can now put our results together so far.

\begin{proof}[Proof of Theorem \ref{thm:coflow.2} for $\epsilon=1$]
Lemma \ref{lem:X.bound} shows that $X$ is bounded away from $0$ and bounded above if $Y$ is bounded.  Lemma \ref{lem:1.Y.bound}, Lemma \ref{lem:1.Y.bound.2} and Proposition \ref{prop:1.Y.unbounded} show that $Y$ is bounded and hence $X$ is bounded.  Finally, Lemma \ref{lem:1.Y.bound.0} shows that $Y$ is bounded away from $0$.  
The observations in \S\ref{sss:strategy} then give the   result.
\end{proof}

\subsubsection{Bounds on $Y$: $\epsilon=-1$} Having proved Theorem \ref{thm:coflow.2} for $\epsilon=1$ now move on to the case where $\epsilon=-1$.  The arguments here are similar to the $\epsilon=1$ case, but often easier.

\begin{lemma}\label{lem:-1.Y.bound}
When $\epsilon=-1$, $Y(s)$ can only diverge in the region to the left of $\gamma_Y$ and above the line $Y=X$, and can only tend to $0$ to the right of $\gamma_Y$ but below the line $Y=X$.
\end{lemma}

\begin{proof}
These are elementary observations given that $Y$ is decreasing to the right of $\gamma_Y$ and increasing to the left of $\gamma_Y$.
\end{proof}

We now see that the evolution equation \eqref{eq:coflow2} for $Y$ has a useful feature when $\epsilon=-1$.

\begin{lemma}\label{lem:-1.Y.bound.2} When $\epsilon=-1$,
there exists a least $\bar{X}\in(1,2)$ such that $X\leq \bar{X}$ whenever $(X,Y)$ is on $\gamma_Y$.  Hence, $\dot{Y}\leq 0$ and thus $Y$ is bounded whenever $X\geq \bar{X}$.
\end{lemma}

\begin{proof}
This is an elementary calculation, showing that $X$ takes a maximum value on the curve $\gamma_Y$.
\end{proof}
 
We deduce from Lemma \ref{lem:-1.Y.bound} and Lemma \ref{lem:-1.Y.bound.2} that the only way $Y$ can become unbounded when $\epsilon=-1$ is if $X$ remains in the interval $(0,\bar{X})$.  We now show that this is not possible.

\begin{proposition}\label{prop:-1.Y.bound}  Recall $\bar{X}$ from Lemma \ref{lem:-1.Y.bound.2}. 
For $\epsilon=-1$, there are no solutions $(X(s),Y(s))$ to the left of $\gamma_Y$ and above $Y=X$ with $X(s)\in (0,\bar{X})$ and $Y(s)$ unbounded.
\end{proposition}

\begin{proof}
We suppose, for a contradiction, that there is such a solution.  Suppose that the solution enters the part of the region where $X\geq 1$.  Then $X$ is strictly increasing, so $X>1$ for all subsequent times.  However, the line $\gamma_Y$ asymptotes to $X=1$ as $Y\to\infty$, so we must eventually have that $X$ is decreasing, which is a contradiction.  

We deduce that   $X(s)\in(0,1)$ for all $s$.
We note that there is a finite $\bar{Y}>0$ such that
\begin{equation}
\begin{split}
\label{eq:coflow.-1.terms}
    Y^2-2(2X^2-2X-1)XY-2X^2(2X-1)(X+1)>0\\
    \text{and}\quad (2X^2-3X-1)Y-2X(1-2X)>0
\end{split}
\end{equation}
for all $X\in(0,1)$ and $Y>\bar{Y}$.
Since $Y$ is increasing in the region we are studying and we are assuming it is unbounded, we may restrict to the case where $Y(s)>\bar{Y}$.  Comparing \eqref{eq:coflow.-1.terms} to \eqref{eq:coflow1}--\eqref{eq:coflow2} yields the differential inequalities
\[
\dot{X}\geq \frac{4}{X}Y^2\quad\text{and}\quad \dot{Y}\leq \frac{8(2-X)Y^3}{X^2}.
\]
We deduce that
\[
\frac{\d Y}{\d X}\leq \frac{2(2-X)}{X}Y
\]
and so, by Gr\"onwall's inequality, we have that $Y$ is bounded by $X^4e^{-2X}$ (up to multiplicative factors and additive constants depending only on the initial conditions).  Since $X\in(0,1)$, this forces our required contradiction.
\end{proof}

To complete the proof in the $\epsilon=-1$ we now only need to show that $Y$ stays away from $0$.

\begin{lemma}\label{lem:-1.Y.bound.0}
When $\epsilon=-1$, $Y(s)$ is bounded away from zero.
\end{lemma}

\begin{proof}
The proof is entirely analogous to that of Lemma \ref{lem:1.Y.bound.0}.  
Remark \ref{rmk:coflow.collpase}  and Lemma \ref{lem:-1.Y.bound} show that we may restrict attention to points near $(1/2,0)$ to the right of $\gamma_Y$.  Writing $X=\frac{1}{2}+\delta X_1$ and $Y=\delta Y_1$ for $\delta$ small, we find that the linear term in $\delta$ gives
\[
\dot{X_1}-\dot{Y_1} = -24(X_1-Y_1) \quad\text{and}\quad \dot{Y_1}=0.
\]
Note that the line $X_1=Y_1$ is tangent to $\gamma_Y$ at $(1/2,0)$ and that $\gamma_Y$ lies above this line.  Therefore, we need only consider $X_1-Y_1>0$ and see that, to leading order, $Y_1$ remains a non-zero constant whilst $X_1-Y_1$ exponentially decreases.  Hence, $Y$ will  be bounded away from $0$.
\end{proof}

We may now complete the proof of Theorem \ref{thm:coflow.2}.

\begin{proof}[Proof of Theorem \ref{thm:coflow.2} for $\epsilon=-1$]
We first see that Lemma \ref{lem:X.bound} shows that $X$ is bounded away from $0$ and bounded if $Y$ is bounded.  Lemmas \ref{lem:-1.Y.bound} and \ref{lem:-1.Y.bound.2}, together with Proposition \ref{prop:-1.Y.bound}, show that $Y$ is bounded, and thus $X$ is also bounded.  Finally, Lemma \ref{lem:-1.Y.bound.0} shows that $Y$ is bounded away from zero, which completes the proof by the discussion in \S\ref{sss:strategy}.
\end{proof}

\section{Laplacian flow}

We now consider the Laplacian flow for our family of $\GG_2$-structures in \eqref{eq:varphi.ansatz}.  We recall that this flow, if it is well-posed and stays within the ansatz, is given by
\begin{equation}\label{eq:flow.ansatz}
    \frac{\partial}{\partial t}\varphi_{\epsilon}(t)=\Delta_{\varphi_{\epsilon}(t)}\varphi_{\epsilon}(t)=\mathrm{d}^*_{\varphi(t)}\mathrm{d}\varphi_{\epsilon}(t),
\end{equation}
for the coclosed 3-forms $\varphi_{\epsilon}(t)$ in \eqref{eq:varphi.ansatz}.  Here we have to be particularly mindful that the coclosed condition may not be preserved by the flow, let alone the ansatz. 

\subsection{The flow equations} We first observe that the Hodge Laplacian of the 3-form defining the $\mathrm{G}_2$-structure follows immediately from that of the 4-form by taking the Hodge star: 
\[
\mathrm{d}\mathrm{d}^*_{\varphi_{\epsilon}}\varphi_{\epsilon}=*\Delta_{\psi_{\epsilon}}\psi_{\epsilon}.
\]
To compute the Hodge star we use the following relations: 
\begin{align*}
*_{\varphi_{\epsilon}}\pi^*\mathrm{vol}_N&=\frac{\epsilon a^2 b}{c^4}\eta_1\wedge\eta_2\wedge\eta_3;\\
*_{\varphi_{\epsilon}}(\eta_2\wedge\eta_3\wedge\omega_1+\eta_3\wedge\eta_1\wedge\omega_2)&= \frac{\epsilon}{b}(\eta_1\wedge\omega_1+\eta_2\wedge\omega_2);\\
*_{\varphi_{\epsilon}}(\eta_1\wedge\eta_2\wedge\omega_1)&=\frac{\epsilon b}{a^2}\eta_3\wedge\omega_3.
\end{align*}
We can now use Lemma \ref{lem:Lap.psi} to find an expression for the Hodge Laplacian of $\varphi_{\epsilon}$, which we need to consider the Laplacian flow \eqref{eq:flow.ansatz}.

\begin{lemma}\label{lem:Lap.phi}
The Hodge Laplacian of $\varphi_{\epsilon}$ in \eqref{eq:varphi.ansatz} is given by:
\begin{equation}\label{eq:Lap.phi}
\begin{split}
\Delta_{\varphi_{\epsilon}}\varphi_{\epsilon} &= \frac{8\epsilon a^2b}{c^2}\left(\frac{2a^2}{c^2} +\frac{b^2}{c^2} +2 + \frac{2\epsilon b}{a} - \frac{b^2}{a^2}\right) \eta_1\wedge\eta_2\wedge\eta_3 \\
&\qquad - 4\Big(\epsilon b+\frac{4 a^3 }{c^2} + \frac{2\epsilon a^2b}{c^2} + \frac{2 c^2}{a} - \frac{\epsilon b c^2}{a^2} \Big) (\eta_{1}\wedge\omega_{1}+ \eta_{2}\wedge\omega_{2} ) \\
&\qquad - 4\epsilon b\Big(2 -\frac{b^2}{a^2} + \frac{2c^2}{a^2}+\frac{4\epsilon a b}{c^2} + \frac{2b^2}{c^2} -\frac{2\epsilon bc^2}{a^3} + \frac{b^2c^2}{a^4} \Big) \eta_{3}\wedge\omega_{3}
\end{split}
\end{equation}
In particular, \eqref{eq:Lap.phi} is coclosed and in the same form as \eqref{eq:varphi.ansatz}, so the Laplacian flow \eqref{eq:flow.ansatz} is well-defined.
\end{lemma}

\begin{remark}
We observe that Lemma \ref{lem:Lap.phi} shows that the coclosed condition on the $\GG_2$-structure is preserved along the Laplacian flow in this situation.
\end{remark}

As a consequence of Lemma \ref{lem:Lap.phi} and \eqref{eq:varphi.ansatz}, we can write down the Laplacian flow \eqref{eq:flow.ansatz} as a system of ODEs for $a(t)$, $b(t)$, $c(t)$:
\begin{align*}
    \frac{\mathrm{d}}{\mathrm{d} t}(a^2b)&=\frac{8 a^2b}{c^2}\left(\frac{2a^2}{c^2} +\frac{b^2}{c^2} +2 + \frac{2\epsilon b}{a} - \frac{b^2}{a^2}\right);\\
    \frac{\mathrm{d}}{\mathrm{d} t}(ac^2)&= 4\Big(\epsilon b+\frac{4 a^3 }{c^2} + \frac{2\epsilon a^2b}{c^2} + \frac{2 c^2}{a} - \frac{\epsilon b c^2}{a^2} \Big);\\
    \frac{\mathrm{d}}{\mathrm{d} t}(bc^2)&=4  b\Big(2 -\frac{b^2}{a^2} + \frac{2c^2}{a^2}+\frac{4\epsilon a b}{c^2} + \frac{2b^2}{c^2} -\frac{2\epsilon bc^2}{a^3} + \frac{b^2c^2}{a^4} \Big).
\end{align*}
Just as for the Laplacian coflow, it turns out that the Laplacian flow is easier to analyze if we introduce scale-invariant quantities and a new time parameter.  The resulting equations we obtain then describe the rescaled Laplacian flow.

\begin{lemma}
Define
\[
X=\frac{a^2}{c^2} \quad\text{and}\quad Y=\frac{ab}{c^2}
\]
and introduce a new variable $s$ by 
\[
\frac{\mathrm{d} s}{\mathrm{d} t}=\frac{1}{c^2}.
\]
If we let $\dot{X}=\frac{\mathrm{d} X}{\mathrm{d} s}$ and  $\dot{Y}=\frac{\mathrm{d} Y}{\mathrm{d} s}$, then the Laplacian flow equations \eqref{eq:flow.ansatz} imply that
\begin{align}
    \dot{X} %&= 4X\left(4X+4\epsilon \frac{Y}{X} -\frac{Y^2}{X^2}+2-\frac{2}{X}-4\epsilon Y+2\epsilon\frac{Y}{X^2}-\frac{Y^2}{X^3}\right)\\
%&=    \frac{4}{X^2}\left(4X^4+4\epsilon X^2Y-XY^2+2X^3-2X^2-4\epsilon X^3Y+2\epsilon XY-Y^2\right)\\
&=\frac{4}{X^2}\left(2X^2(2X-1)(X+1)+2\epsilon (1+2X-2X^2)XY-(1+X)Y^2\right);\label{eq:flow1}\\
    \dot{Y} %&=4\left(2\frac{Y^3}{X}+4Y+3\epsilon\frac{Y^2}{X}-2\frac{Y^3}{X^2}-2\epsilon Y^2-2\frac{Y}{X}+\epsilon\frac{Y^2}{X^2}\right)\\
    %&= \frac{4Y}{X^2}\left(2XY^2+4X^2+3\epsilon XY-2Y^2-2\epsilon X^2Y-2X+\epsilon Y\right)\\
    &=\frac{4Y}{X^2}\left(2X(2X-1)+\epsilon (1+3X-2X^2)Y+2(X-1)Y^2\right)
    .\label{eq:flow2}
\end{align}
\end{lemma}

\subsection{Critical points and dynamics}
We observe that \eqref{eq:flow1}--\eqref{eq:flow2} are just the negative of the equations \eqref{eq:coflow1}--\eqref{eq:coflow2} arising from the Laplacian coflow.  We therefore have the following.

\begin{lemma}
The only critical points for $X,Y>0$ to \eqref{eq:flow1}--\eqref{eq:flow2} are:
\[X=Y=\frac{1}{5}\quad\text{and}\quad \epsilon=1
\]
and
\[
X=Y=1\quad\text{and}\quad \epsilon=-1,
\]
and the condition $X=Y$ is preserved for $\epsilon=1$, but not preserved for $\epsilon=-1$ except when $X=Y=1$.
\end{lemma}

 The observation that the Laplacian flow equations \eqref{eq:flow1}--\eqref{eq:flow2} are the negative of the Laplacian coflow equations also implies the following  result, based on our stability analysis for the Laplacian coflow, which gives Theorem \ref{thm:flow} in the Introduction.

\begin{theorem}
The only critical points of the  Laplacian flow \eqref{eq:flow.ansatz}, after rescaling,  are the nearly parallel $\GG_2$ structures $\varphi^{ts}$ and $\varphi^{np}$ inducing the 3-Sasakian and squashed Einstein metrics.  Both critical points are unstable sources under the rescaled Laplacian flow. 
\end{theorem}

For completeness, we provide the dynamic plots in Figure \ref{fig:flow.plots} for the Laplacian flow \eqref{eq:flow.ansatz} for our ansatz with $\epsilon=1$ and $\epsilon=-1$.  We again indicate the curves $\gamma_X$ and $\gamma_Y$ where $\dot{X}$ and $\dot{Y}$ change sign, respectively.  Of course, the dynamics are simply the opposite of those which appear in the Laplacian coflow plot in Figure \ref{fig:coflow.plots}

\begin{figure}[ht]\caption{Dynamic plots for Laplacian flow for $\epsilon=1$ and $\epsilon=-1$}
    \centering
    \includegraphics[width=0.45\textwidth]{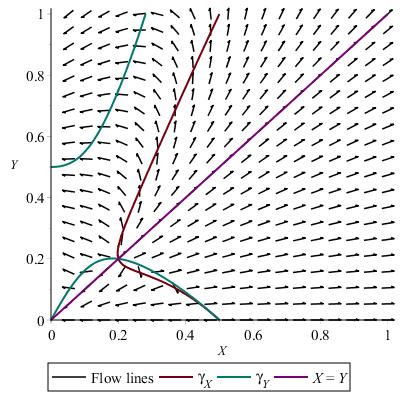}\qquad
     \includegraphics[width=0.45\textwidth]{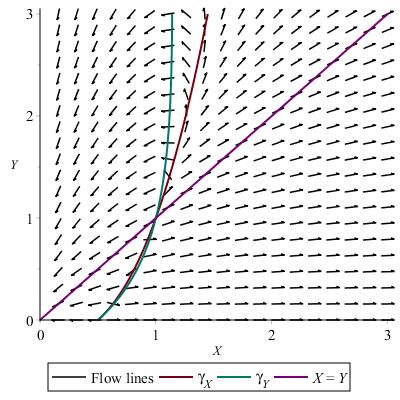}
    \label{fig:flow.plots}
\end{figure}

\begin{remark}
Figure \ref{fig:flow.plots} shows that several different behaviours are possible for the rescaled Laplacian flow when the initial condition is close to a nearly parallel $\GG_2$-structure.  One possibility is flowing to the origin, which corresponds again to the 7-manifold $M$ collapsing to the 4-manifold $N$ through the fibration \eqref{eq:M.Z.pi}.  Notice that the 3-dimensional fibres of \eqref{eq:M.Z.pi} are calibrated by the $\GG_2$-structure $\varphi$, i.e.~$\varphi$ restricts to be the volume form on the fibres, and so are associative by definition.  The fact that $\d\varphi\neq  0$ means that the volume of any compact associative 3-fold is not necessarily topologically determined, unlike for closed $\GG_2$-structures.  %By contrast, in the Laplacian flow for \emph{closed} $\GG_2$-structure,  the fact that the cohomology class of $[\varphi]$ is preserved in this case would mean that any compact associative 3-fold would have constant volume along the flow.
\end{remark}

\section{Ricci flow}

In this section we compare the results we have obtained for the Laplacian flow and coflow of our $\GG_2$-structures with the behaviour of the induced metric of the $\GG_2$-structures under the Ricci flow.  We recall that this comparison is useful because the nearly parallel $\GG_2$-structures, which are critical points of the rescaled Laplacian flow and coflow  that we studied earlier, induce Einstein metrics which are then critical points of the Ricci flow up to scale.  It is also of interest because if we use the Laplacian coflow for coclosed $\GG_2$-structures \eqref{eq:Lap.coflow}, then the induced flow \eqref{eq:Ric.flow.psi} on the induced metric of the $\GG_2$-structures is the Ricci flow plus lower order terms determined algebraically by the torsion of the $\GG_2$-structure.  

\subsection{The flow equations}
We wish to study the Ricci flow for our ansatz
\begin{equation}\label{eq:Ricflow.ansatz}
\frac{\partial}{\partial t}g(t)=-2\mathrm{Ric}(t),
\end{equation}
if it exists. From general theory the Ricci flow will have short time existence starting from our metric ansatz \eqref{eq:metric.ansatz}, though it is not immediately obvious  that the ansatz will be preserved.

We begin by computing the Ricci curvature of   $g(t)$ from \eqref{eq:metric.ansatz}.  (Note that $g(t)$ is independent of   $\epsilon\in\{\pm 1\}$.)

\begin{lemma}\label{lem:Ricci}
Let $g(t)$ be as in \eqref{eq:metric.ansatz}.  The Ricci curvature $\mathrm{Ric}(t)$ of $g(t)$ satisfies
\begin{align*}
\mathrm{Ric} &=2\left(2-\frac{b^2}{a^2}+\frac{2a^4}{c^4}\right)(\eta_1^2+\eta_2^2)+2\left(\frac{b^4}{a^4}+\frac{2b^4}{c^4}\right)\eta_3^2+2\left(6-\frac{2a^2+b^2}{c^2}\right)\pi^*g_N.
\end{align*}
In particular, $\mathrm{Ric}(t)$ is of the same form as  \eqref{eq:metric.ansatz}, and so the Ricci flow \eqref{eq:Ricflow.ansatz} is well-defined.
\end{lemma}

\begin{proof}
Our approach is to use the O'Neill formulas for Riemannian submersions \cite{O}. 

We let $\nabla$ denote the Levi-Civita connection of the 3-Sasakian metric $g_M$.  Recall the orthonormal Killing fields $E_1,E_2,E_3$ in Definition \ref{dfn:3Sak} and let $B_0,B_1,B_2,B_3$ be local orthonormal vector fields on $M$ which are horizontal for the Riemannian submersion \eqref{eq:M.Z.pi}.  It is then straightforward to compute that
\begin{align*}
    \nabla_{E_i}E_j=\sum_{k=1}^3\epsilon_{ijk}E_k, \quad \nabla_{E_i}B_l=0,
    \quad \nabla_{B_l}E_i=-\sum_{m=0}^4\sigma_{ilm}B_m,\quad (\nabla_{B_l}B_m)^{\mathcal{V}}=\sum_{i=1}^3\sigma_{ilm}E_i,
\end{align*}
where $()^{\mathcal{V}}$ indicates the vertical projection with respect to \eqref{eq:M.Z.pi}, $\epsilon_{ijk}$ is the sign of the permutation $(i,j,k)$ of $(1,2,3)$, and $\sigma_{ilm}$ is skew-symmetric in $l,m$ satisfying $\sigma_{i0j}=1$ and $\sigma_{ijk}=\epsilon_{ijk}$ for $i,j,k\in\{1,2,3\}$.  In particular, we notice that the fibres of the Riemannian submersion \eqref{eq:M.Z.pi} are totally geodesic, so the O'Neill tensor often denoted $T$ vanishes, and that the other O'Neill tensor, often called $A$, is horizontally divergence-free. 

We then let $\tilde{\nabla}$ denote the Levi-Civita connection of $g=g(t)$ and let
\begin{gather*}
    \tilde{E}_1=\frac{E_1}{a},\quad \tilde{E}_2=\frac{E_2}{a},\quad \tilde{E}_3=\frac{E_3}{b},\quad \tilde{B}_l=\frac{B_l}{c}.
\end{gather*}
We may then compute that the quantities we need to complete our computation are:
\begin{gather*}
[\tilde{E}_2,\tilde{E}_3]=\frac{2}{b}\tilde{E}_1,\quad [\tilde{E}_3,\tilde{E}_1]=\frac{2}{b}\tilde{E}_2,\quad [\tilde{E}_1,\tilde{E}_2]=\frac{2b}{a^2}\tilde{E}_3,\\
\tilde{\nabla}_{\tilde{E}_2}\tilde{E}_3=  \frac{b}{a^2}\tilde{E}_1,\quad  \tilde{\nabla}_{\tilde{E}_3}\tilde{E}_1=\frac{2a^2-b^2}{a^2b}\tilde{E}_2,\quad  \tilde{\nabla}_{\tilde{E}_1}\tilde{E}_2=\frac{b}{a^2}\tilde{E}_3,\quad \tilde{\nabla}_{\tilde{E}_i}\tilde{B}_l=0,\\
 \tilde{\nabla}_{\tilde{B}_l}\tilde{E}_1=-\frac{a}{c^2}\sum_{m=0}^4\sigma_{1lm}\tilde{B}_m,\quad \tilde{\nabla}_{\tilde{B}_l}\tilde{E}_2=-\frac{a}{c^2}\sum_{m=0}^4\sigma_{2lm}\tilde{B}_m,\quad \tilde{\nabla}_{\tilde{B}_l}\tilde{E}_3=-\frac{b}{c^2}\sum_{m=0}^4\sigma_{3lm}\tilde{B}_m.
\end{gather*}
Using the fact that the metric on $N$ is Einstein with scalar curvature $48$ and the O'Neill formulas, we see that the Ricci tensor of $g$ is diagonal and satisfies:
\begin{align*}
\mathrm{Ric}(\tilde{E}_1,\tilde{E}_1)=\mathrm{Ric}(\tilde{E}_2,\tilde{E}_2)&=\frac{4}{a^2}-\frac{2b^2}{a^4}+\frac{4a^2}{c^4},\\
\mathrm{Ric}(\tilde{E}_3,\tilde{E}_3)&=\frac{2b^2}{a^4}+\frac{4b^2}{c^4},\\
\mathrm{Ric}(\tilde{B}_l,\tilde{B}_l)&=\frac{12c^2-2(2a^2+b^2)}{c^4}
\end{align*}
as desired.
\end{proof}

\noindent Given Lemma  \ref{lem:Ricci}, we can now write down the Ricci flow equations for our ansatz in \eqref{eq:metric.ansatz}:
\begin{align*}
    \frac{\mathrm{d}}{\mathrm{d} t}(a^2)&= -4\left(2-\frac{b^2}{a^2}+2\frac{a^4}{c^4}\right);%\label{eq:a2}
    \\
        \frac{\mathrm{d}}{\mathrm{d} t}(b^2)&= -4\left(\frac{b^4}{a^4}+2\frac{b^4}{c^4}\right);%\label{eq:b2}
        \\
     \frac{\mathrm{d}}{\mathrm{d} t}(c^2)&= -4\left(6-2\frac{a^2}{c^2}-\frac{b^2}{c^2}\right).%\label{eq:c2}       
\end{align*}
To simplify the analysis of these equations we introduce some scale-invariant quantities and rescale the time parameter to  find the following after a short computation.

\begin{lemma}\label{lem:Ric.flow}
Define
\begin{equation*}
    A=\frac{a^2}{c^2}\quad\text{and}\quad B=\frac{b^2}{c^2}
\end{equation*}
and introduce a new variable $s$ by 
\[
\frac{\mathrm{d} s}{\mathrm{d} t}=\frac{1}{c^2}.
\]
If we let $\dot{A}=\frac{\mathrm{d} A}{\mathrm{d} s}$ and $\dot{B}=\frac{\mathrm{d} B}{\mathrm{d} s}$, then the Ricci flow equations for \eqref{eq:metric.ansatz} imply that
\begin{align}
    \dot{A}&=\frac{4(1-A)}{A}\left(B(1+A)-2A(1-2A)\right);\label{eq:Ric.flow.A}
\\
\dot{B}&=\frac{4B}{A^2}\left(2A^2(3-A)-B(1+3A^2)\right).\label{eq:Ric.flow.B}
\end{align}
\end{lemma}

\noindent The equations \eqref{eq:Ric.flow.A}--\eqref{eq:Ric.flow.B} describe the rescaled Ricci flow.

\subsection{Critical points and dynamics}  It is straightforward to find the critical points and observe some basic facts about the dynamics of the flow equations \eqref{eq:Ric.flow.A}--\eqref{eq:Ric.flow.B} as follows.

\begin{lemma}\label{lem:Ric.crit}
The only critical points for $A,B>0$ of the flow equations \eqref{eq:Ric.flow.A}--\eqref{eq:Ric.flow.B} are 
$$A=B=\frac{1}{5}\quad\text{and}\quad A=B=1.$$ 
Moreover, the lines $A=B$ and $A=1$ are preserved by the flow.
\end{lemma}

Lemma \ref{lem:Ric.flow} leads us to draw the plot in Figure \ref{fig:Ric} showing the dynamics of the equations. If we denote the curve where $B(1+A)=2A(1-2A)$ for $A,B\geq 0$ by $\gamma_A$ and the curve $2A^2(3-A)=B(1+3A^2)$ for $A,B\geq 0$ by $\gamma_B$, these are the curves which, together with the line $A=1$, determine the sign of $\dot{A}$ and $\dot{B}$.  

\begin{figure}[ht]\caption{Dynamic plot for Ricci flow}\label{fig:Ric}
\includegraphics[width=0.5\textwidth]{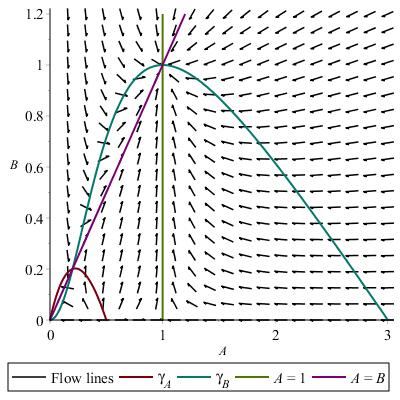} \quad \includegraphics[width=0.35\textwidth]{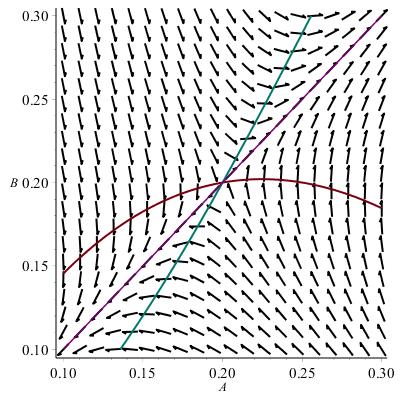}
\end{figure}

The first plot in Figure \ref{fig:Ric} indicates that there is a stable critical point at $A=B=1$, where $\gamma_A$ and $A=1$ intersect. In the second plot, we focus  on the  critical point $A=B=1/5$ where $\gamma_A$ and $\gamma_B$ intersect, which appears to be unstable (in fact, a saddle point).  

\begin{remark}\label{rmk:Ric.collapse.1}
The curves $\gamma_A$ and $\gamma_B$ also intersect at $A=B=0$, which is a degenerate critical point in terms of the equations \eqref{eq:Ric.flow.A}--\eqref{eq:Ric.flow.B}, since they are not defined there.  From the geometric point of view, here $M^7$ has collapsed to the 4-dimensional orbifold $N$ (or a point)  in the fibration \eqref{eq:M.Z.pi}, since $a=b=0$ in \eqref{eq:metric.ansatz}, and so any flow lines tending to $(0,0)$ converge to the Einstein metric $g_N$ on $N$ up to scale (or $0$).  We can avoid the possibility of converging to a point along the rescaled Ricci flow, since the diameter will stay bounded away from $0$.
\end{remark}

\begin{remark}
One may observe  from \eqref{eq:Ric.flow.A}--\eqref{eq:Ric.flow.B} that there are two other degenerate critical points if we allow $A=0$ or $B=0$:
$$(A,B)=\left(\frac{1}{2},0\right)\quad\text{and}\quad (A,B)=(1,0).$$
  These critical points correspond to a collapsed situation where the $S^1$ fibres corresponding to $E_3$ (in the notation of Definition \ref{dfn:3Sak}) now have zero size since $b=0$ in \eqref{eq:ansatz}. The 7-manifold  $M$ has therefore collapsed to the \emph{twistor space} $Z^6$, which is a 2-sphere bundle over $N^4$.  The  metrics with $(A,B)=(1,0)$ and $(A,B)=(1/2,0)$ correspond to two Einstein metrics on the twistor space $Z$: the canonical \emph{K\"ahler--Einstein} and \emph{nearly K\"ahler} metric, respectively.  It is natural to see these Einstein metrics appear at the boundary of our Ricci flow ansatz as critical points.  However, it is interesting to note that the K\"ahler--Einstein metric on $Z$ does not play a distinguished role in the study of the Laplacian coflow and Laplacian flow, whereas the nearly K\"ahler metric does.
\end{remark}

 We now study the dynamics of the flow and show the following, recalling the fibration \eqref{eq:M.Z.pi} of $M^7$ over a 4-dimensional base $N^4$.
 
\begin{theorem}\label{thm:Ric.2} The only critical points for the Ricci flow \eqref{eq:Ricflow.ansatz}, after rescaling, are the 3-Sasakian  metric $g_M$ and the squashed Einstein metric $\tilde{g}_M$ on $M^7$.
The 3-Sasakian metric is a stable critical point for the rescaled Ricci flow within the ansatz \eqref{eq:metric.ansatz}, whilst the squashed Einstein metric is an unstable critical point which is a saddle point. Moreover, there is an open set of initial metrics in the ansatz \eqref{eq:metric.ansatz}, which can be chosen arbitrarily near $\tilde{g}_M$, such that they flow either to $g_M$ or  to the collapsed limit (even after rescaling) where the 3-dimensional fibres in \eqref{eq:M.Z.pi} shrink to $0$ and  the flow converges to the Einstein 4-orbifold $(N,g_N)$.
\end{theorem}
 
 \begin{proof}  We recall the curves $\gamma_A$, $\gamma_B$ introduced after Lemma \ref{lem:Ric.crit} and plotted in Figure \ref{fig:Ric}.
 
For $A>1$ we have that $\dot{A}<0$ and hence $A$ remains bounded as long as the flow exists. We also see that $\dot{B}<0$ whenever $B$ is above the curve $\gamma_B$ (which passes through $(0,0)$ and $(3,0)$) and hence also remains bounded as long as the flow exists.  Since
\[
\frac{\d}{\d t}(c^2)=-4(6-2A-B),
\]
the right-hand side is bounded and so $c^2$ cannot blow up in finite time.  Moreover, $c^2$ goes to zero at a linear rate in $t$ at most, and so we can integrate $\frac{1}{c^2}$ with respect to $t$ to obtain $s$ in Lemma \ref{lem:Ric.flow}. 

We linearize  \eqref{eq:Ric.flow.A}--\eqref{eq:Ric.flow.B} around $(A,B)=(1,1)$ to obtain:
$$\dot{A}=-16A\quad\text{and}\quad \dot{B}=-16B.$$
Hence $(1,1)$ is a stable critical point, which corresponds to the 3-Sasakian  metric.  

We now observe that if we start above the curve $\gamma_A$, which passes through $(0,0)$ and $(\frac{1}{2},0)$, then   $\dot{A}>0$ for $A<1$, so $A$ is always bounded away from $0$ in this region, depending on its initial value.  Recalling that $\dot{A}<0$ for $A>1$, we deduce that all of the terms in \eqref{eq:Ric.flow.A}--\eqref{eq:Ric.flow.B} are bounded and there can be no periodic orbits in this region as $A$ cannot cross the line $A=1$ by Lemma \ref{lem:Ric.crit}.   We also note that $\dot{B}>0$ when $A\in(1/5,1)$ and $B$ is near $0$ since then $(A,B)$ lies below $\gamma_B$, and thus $B$ is bounded away from $0$  when $A\in(1/5,1)$.  

Since $(1/5,1/5)$ lies on $\gamma_A$, we deduce that the flow will converge to the critical point at $(1,1)$ if the initial value of $(A,B)$ lies above or to the right of the curve $\gamma_A$ and $A$ is greater than $1/5$ initially.  This proves the statement about initial conditions near $\tilde{g}_M$ which flow to $g_M$, since $\tilde{g}_M$ corresponds to $(1/5,1/5)$.  

Now if we specialize to the line $A=B$ which, by Lemma \ref{lem:Ric.flow}, is preserved, we see that
\begin{equation}\label{eq:Ric.AB}
    \dot{A}=4(1-A)(5A-1).
\end{equation}
We see immediately that for $A=B<1/5$ we have $\dot{A}<0$, for $A=B\in (1/5,1)$ we have $\dot{A}>0$. (For $A=B>1$ we have $\dot{A}<0$ again, as expected by the stability of $(1,1)$.) Thus, the critical point at $(1/5,1/5)$ is unstable, even within the restricted ansatz when $A=B$.  Furthermore, if we linearize the system \eqref{eq:Ric.flow.A}--\eqref{eq:Ric.flow.B} around $(1/5,1/5)$ we obtain
\[
\dot{A}=-\frac{16}{5}A+\frac{96}{5}B\quad\text{and}\quad \dot{B}=\frac{192}{5}A-\frac{112}{5}B,
\]
from which it follows that $(1/5,1/5)$ is a saddle point, as the matrix corresponding to the above dynamical system has one positive and one negative eigenvalue: $16$ and $-208/5$.  

We see from \eqref{eq:Ric.AB} that we have initial conditions for \eqref{eq:Ric.flow.A}--\eqref{eq:Ric.flow.B}, even with $A=B$, such that   $A$ and $B$ go to $0$.  In fact, suppose we choose any initial condition in the region $\mathcal{R}$ below $\gamma_A$ but above $\gamma_B$, which means in particular that $A,B\in(0,\frac{1}{5})$.  Since the flow lines enter $\mathcal{R}$ vertically from above along $\gamma_A$ and horizontally from the right along $\gamma_B$, no flow lines can leave $\mathcal{R}$ so $A,B$ are bounded and can only reach $0$ when $(A,B)=(0,0)$.  In $\mathcal{R}$, $\dot{A}<0$ and $\dot{B}<0$ and  there are no periodic orbits.  We conclude that all flow lines starting in $\mathcal{R}$ must converge to $(0,0)$.   We again note as in Remark \ref{rmk:Ric.collapse.1} that $M$ cannot collapse to a point along the rescaled Ricci flow.  The discussion in Remark \ref{rmk:Ric.collapse.1} then implies that the point $(0,0)$ corresponds to the fibres of the fibration \eqref{eq:M.Z.pi} collapsing, even in the rescaled Ricci flow, so that $M^7$ collapses to the base $N^4$ (with its Einstein metric).  Since $\mathcal{R}$ is open and has $(1/5,1/5)$ on its boundary, this complete the proof.
\end{proof}

%\begin{remark}
%One may observe that there are two other degenerate critical points: when $B=0$ and $A$ is either $1/2$ or $1$, as one can see from \eqref{eq:Ric.flow.A}--\eqref{eq:Ric.flow.B}.  These critical points correspond to a collapsed situation where the $S^1$ fibres corresponding to $E_3$ (in the notation of Definition \ref{dfn:3Sak}) now have zero size, as $M$ has therefore collapsed to a 6-manifold $N^6$ which is a 2-sphere bundle over $Z^4$.  As we said earlier, $N^6$ is the \emph{twistor space} of $N$, and the metrics with $(A,B)=(1,0)$ and $(A,B)=(1/2,0)$ correspond to two Einstein metrics on the twistor space: the canonical \emph{K\"ahler--Einstein} and \emph{nearly K\"ahler} metric, respectively.  It is natural to see these Einstein metrics appear at the boundary of our Ricci flow ansatz as critical points.  However, it is interesting to note that the K\"ahler--Einstein metric on $N$ does not play a distinguished role in the study of the Laplacian coflow and Laplacian flow, whereas the nearly K\"ahler metric does.
%\end{remark} 

\noindent Theorem \ref{thm:Ric.2} proves Theorem \ref{thm:Ric} in the Introduction.

\begin{remark}
    By dynamical systems theory, there is a 1-dimensional stable manifold for the squashed Einstein metric within our rescaled Ricci flow ansatz.  We can discern this stable manifold from the plots in Figure \ref{fig:Ric}.  It might be interesting to
understand whether this stable manifold (or, equally, the corresponding unstable manifold)  has any geometric significance, e.g.~any special curvature properties. 
\end{remark}

\paragraph{\bf Acknowledgements.} AK would like to thank Rafe Mazzeo and his advisor Dave Morrison for useful suggestions related to this paper. JDL would like thank the Simons Laufer
Mathematical Sciences Institute (previously known as MSRI) in Berkeley, California, for hospitality during the latter stages of this project. The authors also thank the Simons Collaboration on Special Holonomy in Geometry, Analysis, and Physics for support and many interesting talks which inspired this research.


\begin{thebibliography}{99}
\bibitem{AF}  I.~Agricola and Th.~Friedrich, {\it 3-Sasakian manifolds in dimension seven, their spinors and $\mathrm{G}_2$-structures}, J.~Geom.~Phys.~{\bf 60} (2010), no.~2, 326--332.
\bibitem{Bar} C. B\"ar, \textit{Real Killing spinors and holonomy},   Comm.~Math.~Phys.~{\bf 154} (1993), no.~3, 509--521. 
\bibitem{BG} C.~Boyer and K.~Galicki, {\it 3-Sasaki manifolds}, Surveys in Differential Geometry, vol. 6; Essays on
Einstein Manifolds, 2001.
\bibitem{BGBook} C. Boyer and K. Galicki,  \textit{Sasakian geometry}, Oxford Mathematical Monographs, Oxford University Press, Oxford, 2007.
\bibitem{Bryant} R.L.~Bryant, \textit{Some remarks on $\mathrm{G}_2$-structures}, Proceedings of G\"okova Geometry-Topology Conference 2005, 75--109, G\"okova Geometry/Topology Conference (GGT), G\"okova, 2006.
\bibitem{CN} D.~Crowley and J.~Nordstr\"om, \textit{New invariants of $\mathrm{G}_{2}$-structures}, Geom.~Topol.~{\bf 19} (2015), no.~5, 2949--2992.
\bibitem{FKMS} Th.~Friedrich, I.~Kath, A.~Moroianu and U.~Semmelmann, {\it On nearly parallel $\mathrm{G}_2$-structures}, J. Geom. Phys.~{\bf 23} (1997), no.~3-4, 259--286.
\bibitem{GS} K.~Galicki and S. ~Salamon, \textit{Betti numbers of 3-Sasakian manifolds}, Geom. Dedicata ~{\bf 63} (1996), 45--68.
\bibitem{Grigorian} S.~Grigorian,  {\it Short-time behaviour of a modified Laplacian coflow of $\mathrm{G}_2$-structures}, Adv. ~Math.~{\bf 248} (2013), 378--415.
\bibitem{JoyceBook} D.~Joyce, {\it Compact manifolds with special holonomy}, Oxford Mathematical Monographs, Oxford University Press, Oxford, 2000.
\bibitem{KLL} S.~Karigiannis, N.C.~Leung, and J.D.~Lotay (eds.), {\it Lectures and surveys on $\mathrm{G}_2$-manifolds and related topics}, Springer, 2020.
\bibitem{KMT} S.~Karigiannis, B.~McKay and M.-P.~Tsui, \textit{Soliton solutions for the Laplacian co-flow of some $\rm G_2$-structures with symmetry}, Differential Geom.~Appl.~{\bf 30} (2012), no.~4, 318--333.
\bibitem{Kroncke} K.~Kr\"oncke, {\it Stability and instability of Ricci solitons}, Calc.~Var.~Partial Differential Equations {\bf 53} (2015), no.~1-2, 265--287.
\bibitem{Lehmann} F.~Lehmann, \textit{Deformations of asymptotically conical $\mathrm{Spin}(7)$-Manifolds},  arXiv:2101.10310.
\bibitem{LO} J.D.~Lotay and G.~Oliveira, {\it Examples of deformed $\mathrm{G}_2$-instantons/Donaldson--Thomas connections}, Ann.~Inst.~Fourier (Grenoble) {\bf 72} (2022), no.~1, 339--366. 
\bibitem{LSS} J.D.~Lotay, H.N.~S\'a Earp and J.~Saavedra,  {\it Flows of $\rm G_2$-structures on contact Calabi--Yau 7-manifolds}, Ann.~Global Anal.~Geom.~{\bf 62} (2022), no.~2, 367--389.
\bibitem{LotayWei1} J.D.~Lotay and Y.~Wei, {\it Laplacian flow for closed ${\rm G}_2$ structures: Shi-type estimates, uniqueness and compactness}, Geom.~Funct.~Anal.~{\bf 27} (2017), no.~1, 165--233.
\bibitem{LotayWei2} J.D.~Lotay and Y.~Wei, {\it Stability of torsion-free $\rm G_2$ structures along the Laplacian flow}, J.~Differential Geom.~{\bf 111} (2019), no.~3, 495--526. 
\bibitem{SWW} U.~Semmelmann, C.~Wang and M.Y.-K.~Wang, \textit{Linear instability of Sasaki Einstein and nearly parallel $\mathrm{G}_2$ manifolds}, Internat.~J.~Math.~{\bf 33} (2022), no.~6, Paper No.~2250042, 17 pp.
\bibitem{O} B. ~O'Neill, \textit{The fundamental equations of a submersion}, Michigan Math. J. ~{\bf13} (1966), 459--469.
 
\end{thebibliography}
 \end{document}